\documentclass[11pt,draftcls,onecolumn]{IEEEtran}

\usepackage{latexsym,epsfig,psfrag,color,amssymb,amsmath,amsthm,amsfonts,amstext,graphicx}
\usepackage{subfig}

\long\def\symbolfootnote[#1]#2{\begingroup%
\def\thefootnote{\fnsymbol{footnote}}\footnote[#1]{#2}\endgroup}

\newcommand{\Prob}{\mathrm{P}}
\newcommand{\Expect}{\mathrm{E}}
\newcommand{\indic}{\mathbb{I}}
\newcommand{\defeq}{\overset{\Delta}{=}}
\newcommand{\hLam}{\tilde{\Lambda}}

\newtheorem{remark}{Remark}
\newtheorem{theorem}{Theorem}
\newtheorem{lemma}{Lemma}
\newtheorem{algorithm}{Algorithm}

\begin{document}
\title{Data-Efficient Quickest Change Detection with On-Off Observation Control}
\author{\IEEEauthorblockN{Taposh Banerjee and Venugopal V. Veeravalli}\\
\IEEEauthorblockA{ECE Department and
Coordinated Science Laboratory\\
University of Illinois at Urbana-Champaign, Urbana, IL\\
Email: banerje5,vvv@illinois.edu}}

\maketitle


\begin{abstract}
In this paper we extend the Shiryaev's quickest change detection
formulation by also accounting for the cost of observations used
before the change point. The observation cost is captured through the
average number of observations used in the detection process before
the change occurs. The objective is to select an on-off observation
control policy, that decides whether or not to take a given
observation, along with the stopping time at which the change is
declared, so as to minimize the average detection delay, subject to
constraints on both the probability of false alarm and the observation
cost. By considering a Lagrangian relaxation of the constraint
problem, and using dynamic programming arguments, we obtain an
\textit{a posteriori} probability based two-threshold algorithm that
is a generalized version of the classical Shiryaev algorithm. We
provide an asymptotic analysis of the two-threshold algorithm and show
that the algorithm is asymptotically optimal, i.e., the performance of
the two-threshold algorithm approaches that of the Shiryaev algorithm,
for a fixed observation cost, as the probability of false alarm goes
to zero. We also show, using simulations, that the two-threshold
algorithm has good observation cost-delay trade-off curves, and
provides significant reduction in observation cost as compared to the
naive approach of fractional sampling, where samples are skipped
randomly. Our analysis reveals that, for practical choices of
constraints, the two thresholds can be set independent of each other:
one based on the constraint of false alarm and another based on the
observation cost constraint alone.

\symbolfootnote[0]{This research is partially supported by the National Science Foundation under grant CCF 08-30169, through the University of Illinois at Urbana-Champaign.
This research was also supported in part by the U.S. Army Research Office MURI grant W911NF-06-1-0094, through a subcontract from Brown University at the University of Illinois, and by the U.S. Defense Threat Reduction Agency through subcontract 147755 at the University of Illinois from prime award HDTRA1-10-1-0086.}
\end{abstract}

\section{Introduction} \label{sec:Intro}
In the  Bayesian quickest change detection problem proposed by
Shiryaev \cite{Shiryaev63}, there is a
sequence of random variables, $\{X_n\}$, whose distribution changes at a random time $\Gamma$. It is assumed that before $\Gamma$, $\{X_n\}$ are independent and identically distributed (i.i.d.) with density $f_0$, and after $\Gamma$
they are i.i.d.\ with density $f_1$. The distribution of $\Gamma$ is assumed to be known and modeled as a geometric random variable with parameter $\rho$.  The objective is to find a stopping
time $\tau$, at which time the change is declared, such that the average detection delay is minimized subject to a constraint on the probability of false alarm.

In this paper we extend Shiryaev's formulation by explicitly accounting for the cost of the observations used in the detection process.
We capture the observation penalty (cost) through the average number of observations used before the change point $\Gamma$, 
and allow for a dynamic control policy that determines whether or not a given observation is taken. The objective is to choose the
observation control policy along with the stopping time $\tau$, so that the average detection delay is minimized subject to constraints on the probability of false alarm and the observation cost.
The motivation for this model comes from the consideration of the following engineering applications.

In many monitoring applications, for example infrastructure monitoring, environment monitoring, or habitat monitoring, especially of endangered species, surveillance
is only possible through the use of inexpensive battery operated sensor nodes. This could be due to the high cost of employing a wired sensor network or a human observer, or the infeasibility of having a human intervention. For example in habitat monitoring of certain sea-birds as reported in \cite{Mainwaring}, the very reason the birds chose the habitat was because of the absence of humans
and predators around it. In these applications the sensors are typically deployed for long durations, possibility over months, and due the constraint on energy, the most effective way to save energy at the sensors is to switch the sensor between on and off states. An energy-efficient quickest change detection algorithm can be employed here that can operate over months and trigger other more sophisticated and costly sensors, which are possibly power hungry, or more generally, trigger a larger part of the sensor network \cite{JenniferRice}.
 This change could be a fault in the structures in infrastructure monitoring \cite{JenniferRice}, the arrival of the species to the habitat \cite{Mainwaring}, etc.

In industrial quality control, statistical control charts are designed that can detect a sustained deviation of the industrial process from normal behavior \cite{SPC_SOTA}.
Often there is a cost associated with acquiring the statistics for the control charts and it is of interest to consider designing \textit{economic-statistical} control chart schemes
\cite{SPC_SOTA, Tagaras, Reynolds, AssafPollak1993,Yakir1996, Makis2005, Assaf1988}. One approach to economic-statistical control chart design has been to use algorithms from the change detection literature, such as Shewhart, EWMA and CUSUM, as control charts, and optimize over the choice of sample size, sampling
interval and control limits \cite{Tagaras,Reynolds}.
Another approach has been to find optimal sampling rates in the problem of detection of a change
in the drift of a sequence of Brownian motions with global false alarm constraint \cite{AssafPollak1993,Yakir1996}.
Thus, these approaches are essentially non-Bayesian.
It has been demonstrated, mostly through numerical results, that Bayesian control charts, which choose the parameters of the detection algorithms based on the
posterior probability that the process is out of control, perform better than the traditional control charts based on
Shewhart, EWMA or CUSUM; see \cite{Makis2005}, and the references therein.
The problem of dynamic sampling for detecting a change in the drift of a standard Brownian motion is considered for an exponentially distributed change point in \cite{Assaf1988}.
For practical applications, it is of interest to consider the economic design of Bayesian control charts in discrete time.
The design of a Bayesian economic-statistical control chart is considered for a shift in the mean vector of a multivariate Gaussian model in \cite{Makis2005}.
But, the problem is modeled as an optimal stopping problem that minimizes the long term average cost, and hence, there is no control on the number of observations used at each
time step.
The process control problem is fundamentally a quickest change detection problem, and it is therefore appropriate that economic-statistical schemes for process control are developed
in this framework. 

In most of the above mentioned or similar applications, changes are rare and quick detection is often required.
So, 
ideally we would like to
take as few observations as possible before change to reduce the observation cost, and skip as few as possible after change to minimize delay, while maintaining an acceptable probability of false alarm.

There have been other formulations of the Bayesian quickest change detection problem that are relevant to sensor networks: see
\cite{Veeravalli01}-\cite{TaposhVinod}.
The change detection problem studied here was earlier considered in a similar set-up
for sensor networks in \cite{Prem}. But owing to the complexity of the problem, the structure of the optimal policy was studied only numerically, and for the same reason, no analytical expressions were developed for the performance.

The goal of this paper is to develop a deeper understanding of the trade-off between delay, false alarm probability, and the cost of observation or information,
and to identify a control policy for data-efficient quickest change detection that has some optimality property and is easy to design.
We extend the Shiryaev's formulation by also accounting for the cost of observations used before the change point,
and obtain an {\em a posteriori} probability based two-threshold algorithm that is asymptotically optimal. Specifically,
we show that the probability of false alarm and the average detection delay of the two-threshold algorithm approaches that of the Shiryaev algorithm,
for a fixed observation cost constraint, as the probability of false alarm goes to zero.
Even for moderate values of the false alarm probability, we will show using simulations that the two-threshold algorithm provides good performance.
We also provide an asymptotic analysis of the two-threshold algorithm, i.e., we obtain expressions for the delay, probability of false alarm and the average
number of observations used before and after change, using which the thresholds can be set to meet
the constraints on probability of false alarm and observation cost.

The layout of the paper is as follows.
In the following section, we set up the data-efficient quickest change detection problem with on-off observation control and introduce the two-threshold algorithm.
In Section~\ref{sec:Analysis}, we provide an asymptotic analysis of the two-threshold algorithm. In Section~\ref{sec:Approx_AND_Numerical}, we provide approximations
using which the analytical expressions in Section~\ref{sec:Analysis} can be computed, and validate the approximations by comparing them with the corresponding values
obtained via simulations. In Section~\ref{sec:OptimalityofTwoThreshold}, we prove the asymptotic
optimality of the two-threshold algorithm, provide its false alarm-delay-observation cost trade-off curves and also compare its performance with the naive approach of fractional sampling, where observations are skipped randomly.

\section{Problem Formulation and the Two-threshold Algorithm}
\label{sec:Model}
As in  the  model for the classical Bayesian quickest change detection problem described in Section \ref{sec:Intro}, we have a sequence of random variables $\{X_n\}$, which are i.i.d.\ with density $f_0$ before the random change point $\Gamma$, and i.i.d.\ with density $f_1$ after $\Gamma$. The change point $\Gamma$ is modeled as geometric with parameter $\rho$, i.e., for $0 < \rho < 1, \ 0 \leq \pi_0 < 1$,
\[
\pi_k  = \Prob \{\Gamma= k \} = \pi_0\; \indic_{\{k=0\}} + (1-\pi_0) \rho (1-\rho)^{k-1} \; \indic_{\{k\geq 1\}},
\]
where $\indic$ is the indicator function, and $\pi_0$ represents the probability of the change having happened before the observations are taken. Typically $\pi_0$ is set to 0.

In order to minimize the average number of observations used before $\Gamma$, at each time instant, a decision is made on whether to use the observation in the next time step, based on all the
available information. Let $S_k \in \{0, 1\}$, with $S_k = 1$ if it is been decided to take the observation at time $k$, i.e. $X_k$ is available for decision making, and $S_k=0$ otherwise. Thus, $S_k$ is an on-off (binary) control input based on the information available up to time $k-1$, i.e.,
\[
S_k = \mu_{k-1} (I_{k-1}), \quad k = 1, 2, \ldots
\]
with $\mu$ denoting the control law and $I$ defined as:
\[
I_k = \left[ S_1, \ldots, S_k, X_1^{(S_1)}, \ldots, X_k^{(S_k)} \right].
\]
Here, $X_i^{(S_i)}$ represents $X_i$ if $S_i=1$, otherwise $X_i$ is absent from the information vector $I_k$. The choice of $S_1$ is  based on the prior $\pi_0$.

As in the classical change detection problem, the end goal is to choose a stopping time on the observation sequence at which time the change is declared. Denoting the stopping time by $\tau$, we can define the average detection delay (ADD) as
\[
\text{ADD} = \mathrm{E}\left[(\tau - \Gamma)^+\right].
\]
Further, we can define the probability of false alarm (PFA) as
\[
\text{PFA} = \mathrm{P}(\tau<\Gamma).
\]
The new performance metric for our problem is the average number of observations (ANO) used before $\Gamma$ in detecting the change:
\[
\text{ANO} = \mathrm{E}\left[\sum_{k=1}^{\min\{\tau, \Gamma-1\}} S_k\right].
\]

Let $\gamma = \{\tau, \mu_0, \ldots, \mu_{\tau-1} \}$ represent a policy for cost-efficient quickest change detection.  We wish to solve the following optimization problem:
\begin{eqnarray}
\label{eq:basicproblem}
\underset{\gamma}{\text{minimize}} && \text{ADD}(\gamma),\nonumber \\
\text{subject to } \hspace{-0.15cm} &&\text{PFA}(\gamma) \leq \alpha, \text{  and  } \text{ANO}(\gamma) \leq \beta,
\end{eqnarray}
where $\alpha$ and $\beta$ are given constraints. Towards solving \eqref{eq:basicproblem}, we consider a Lagrangian relaxation
of this problem which can be approached using dynamic programming:
%
\begin{equation}
\label{eq:Lagrangian}
J^* = \min_{\gamma} \text{ADD}(\gamma) + \lambda_f \; \text{PFA}(\gamma) + \lambda_e\; \text{ANO}(\gamma),
\end{equation}
where $\lambda_f$ and $\lambda_e$ are  Lagrange multipliers. It is easy to see that if $\lambda_f$ and $\lambda_e$ can be found such that
the solution to \eqref{eq:Lagrangian} achieves the PFA and ANO constraints with equality, then the solution to \eqref{eq:Lagrangian} is also the solution to
\eqref{eq:basicproblem}.

The problem in \eqref{eq:Lagrangian} can be converted to an appropriate Markov control problem using steps similar to those followed in \cite{Prem}.

Let $\Theta_k$ denote the state of the system at time $k$. After the stopping time $\tau$ it is assumed that the system enters a terminal state ${\cal T}$ and stays there. For $k < \tau$, we have  $\Theta_k =0$ for $k < \Gamma$, and $\Theta_k = 1$ otherwise. Then we can write
\[
\text{ADD}  = \Expect\left[ \sum_{k=0}^{\tau-1} \indic_{\{\Theta_k= 1\}} \right]
\]
and $\text{PFA} = \Expect [ \indic_{\{\Theta_\tau = 0\}}]$.

Furthermore, let $D_k$ denote the stopping decision variable at time $k$, i.e., $D_k = 0$ if $k < \tau$ and $D_k = 1$ otherwise. Then the optimization problem in \eqref{eq:Lagrangian} can be written as a minimization of an additive cost over time:
\[
J^* = \min_{\gamma}  \Expect\left[\sum_{k=0}^{\tau}  g_k (\Theta_k, D_k, S_k)\right]
\]
with
\[
g_k (\theta, d, s) = \indic_{\{\theta \neq {\cal T}\} } \; \left[  \indic_{\{\theta=1\}} \indic_{\{d=0\}} + \lambda_f \; \indic_{\{\theta=0\}} \indic_{\{d=1\}} \right.
 \left.  + \lambda_e \; \indic_{\{\theta=0\}} \indic_{\{s=1\}} \indic_{\{d=0\}}\right].
\]
Using standard arguments \cite{Bertsekas} it can be seen that this optimization problem can be solved using infinite horizon dynamic programming with sufficient statistic (belief state) given by:
\[
p_k = \Prob\{ \Theta_k = 1 \; | \; I_k\} = \Prob \{ \Gamma \leq k \; | \; I_k\}.
\]
Using Bayes' rule, $p_k$ can be shown to satisfy the recursion
\[
p_{k+1} = \begin{cases}
\Phi^{(0)} (p_k) & \text{~if~} S_{k+1} = 0\\
\Phi^{(1)} (X_{k+1}, p_k) & \text{~if~} S_{k+1} = 1
\end{cases}
\]
where
\begin{equation} \label{eq:recursionSkip}
\Phi^{(0)} (p_k)  = p_k + (1-p_k) \rho
\end{equation}
and
\begin{equation} \label{eq:recursionTake}
\Phi^{(1)} (X_{k+1},p_k)  =  \frac{\Phi^{(0)} (p_k)   L(X_{k+1})}{ \Phi^{(0)} (p_k)  L(X_{k+1}) + (1-\Phi^{(0)} (p_k) )}
\end{equation}
with $L(X_{k+1}) = f_1(X_{k+1})/f_0(X_{k+1})$ being the likelihood ratio,  and $p_0=\pi_0$. Note that the structure of recursion for $p_k$ is independent of time $k$.

The optimal policy for the problem given in \eqref{eq:Lagrangian} can be obtained from the solution to the Bellman equation:
\begin{equation} \label{eq:Bellman}
J(p_k) = \min_{d_k,s_{k+1}}  \lambda_f \; (1-p_k) \indic_{\{d_k=1\}} + \indic_{\{d_k=0\}} \left[ p_k + A_J (p_k)\right],
\end{equation}
where
\[
A_J (p_k) = B_{0} (p_k) \indic_{\{s_{k+1}=0\}} + (\lambda_e (1-p_k) + B_{1} (p_k)) \indic_{\{s_{k+1}=1\}},
\]
with
\[
B_{0} (p_k) = J (\Phi^{(0)} (p_k))\]
and
\[ B_{1} (p_k) = \Expect[J (\Phi^{(1)} (X_{k+1}, p_k))].
\]
It can be shown by an induction argument (see, e.g.,  \cite{Prem}) that $J$, $B_{0}$ and $B_{1}$ are all non-negative concave functions on the interval $[0,1]$,  and that $J(1) = B_{0} (1) = B_{1} (1) =0$. Also, by Jensen's inequality
\[
B_{1}(p) \leq J (\Expect[\Phi^{(1)} (X, p)]) = B_{0} (p), \quad p \in [0,1].
\]
Let
\[d(p_k) = B_{0} (p_k) - B_{1} (p_k).\]
Then, from the above properties of $J$, $B_{0}$ and $B_{1}$, it is easy to show that  the optimal policy $\gamma^* = (\tau^*, \mu_0^*, \mu_1^*, \ldots, \mu_{\tau-1}^*)$ for  the problem given in \eqref{eq:Lagrangian} has the following structure:
\begin{equation}
\label{eq:OptimalAlgo}
\begin{split}
               S_{k+1}^* & = \mu_k^* (p_k) =  \begin{cases}
               0 & \mbox{ if } d(p_k) < \lambda_e (1-p_k)\\
               1 & \mbox{ if } d(p_k) \geq \lambda_e (1-p_k)
               \end{cases}\\
               \tau^* &= \inf\left\{ k\geq 1: p_k > A^* \right\}.
               \end{split}
\end{equation}
\begin{remark}
Since, $d(p_k)\geq 0 \ \ \forall p_k$, the algorithm in \eqref{eq:OptimalAlgo} reduces to the classical Shiryaev algorithm when $\lambda_e=0$ \cite{Shiryaev63}.
\end{remark}

The optimal stopping rule $\tau^*$ is similar to the one of the Shiryaev problem. But, the observation control is not explicit and one has to evaluate the differential cost function
$d(p_k)$ at $p_k$ at each time step to choose $S_{k+1}$.

\begin{figure}[htb]
  \centering
  \subfloat[$d(p) = B_0(p) - B_1(p)$ and $\lambda_e(1-p)$ as a function of $p$]{\label{fig:BellmanTwoThresholdCurvesB0B1}\includegraphics[width=0.48\textwidth]{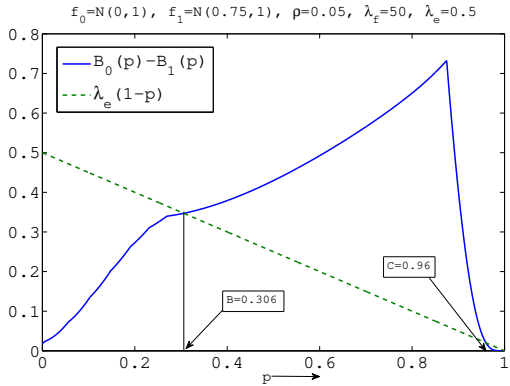}}
  \subfloat[$p+A_J(p)$ and $\lambda_f(1-p)$ as a function of $p$]{\label{fig:BellmanTwoThresholdCurvesAJ}\includegraphics[width=0.48\textwidth]{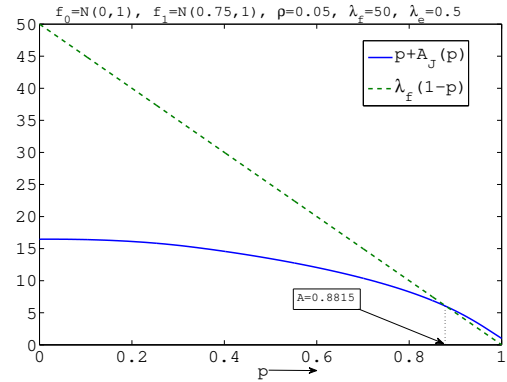}}
  \caption{Example where a two-threshold policy is optimal: $f_0 \sim {\cal N}(0,1)$, $f_1 \sim {\cal N}(0.75,1)$, $\rho=0.05$, $\lambda_f=50$, and $\lambda_e=0.5$. Value iteration: number of iterations=1500, number of points=2000.}
  \label{fig:BellmanTwoThresholdCurves}
\end{figure}
In Fig. \ref{fig:BellmanTwoThresholdCurvesB0B1} we plot the
differential cost function $d(p) = B_0(p) - B_1(p)$ and $\lambda_e(1-p)$ as a function of $p$. We note that, although $B_0(p)$ and $B_1(p)$ are concave in $p$, their
difference $d(p)$ is not. Thus, the line $\lambda_e(1-p)$ can intersect $d(p)$ at more than two points. However, in Fig.~\ref{fig:BellmanTwoThresholdCurvesB0B1} we see that
there are exactly two points of intersection, one at $B=0.306$ an another at $C=0.96$. In Fig.~\ref{fig:BellmanTwoThresholdCurvesAJ} we plot the functions $p+A_J(p)$ and
$\lambda_f(1-p)$ as a function of $p$. This figure shows that the stopping threshold is $A=0.8815<0.96=C$. Thus, from
Fig~\ref{fig:BellmanTwoThresholdCurvesB0B1} and \ref{fig:BellmanTwoThresholdCurvesAJ} we see that the optimal policy has two thresholds. For most of the system parameters
we have tried, the cost functions behave in this way, and hence for these values, the following two-threshold policy is optimal.
\begin{algorithm}[Two-threshold policy: $\gamma(A,B)$]
\label{algo:TwoThreshold}
Start with $p_0=0$ and use the following control, with $B<A$, for $k\geq 0$:
\begin{equation}
\label{eq:TwoThresholdAlgo}
\begin{split}
               S_{k+1} & = \mu_k (p_k) =  \begin{cases}
               0 & \mbox{ if } p_k < B\\
               1 & \mbox{ if } p_k \geq B
               \end{cases}\\
               \tau &= \inf\left\{ k\geq 1: p_k > A \right\}.
               \end{split}
\end{equation}
The probability $p_k$ is updated using \eqref{eq:recursionSkip} and \eqref{eq:recursionTake}.
\end{algorithm}
Extensive numerical studies of the Bellman equation \eqref{eq:Bellman} also shows that there exists choices of $\rho$, $f_0$, $f_1$, $\lambda_f$ and $\lambda_e$ for which
\eqref{eq:TwoThresholdAlgo} is not optimal. In Fig.~\ref{fig:BellmanThreeThresholdCurves} we plot one such case. Note from
Fig.~\ref{fig:BellmanThreeThresholdCurvesB0B1} that again there are two points of intersection of the plotted curves, one at $B=0.9315$ and another at $C=0.973$. But
Fig.~\ref{fig:BellmanThreeThresholdCurvesAJ} shows that $A=0.986>0.973=C$. Thus, the optimal policy has three thresholds. But, note that the value of $\rho=0.7$
is quite large and hence impractical.
Also, simulations with these choices of thresholds show that the ANO is approximately zero. In all the cases we have found, for which the two-threshold policy is not optimal,
the value of $\rho$ is large and ANO is almost zero.
\begin{figure}[htb]
  \centering
  \subfloat[$d(p) = B_0(p) - B_1(p)$ and $\lambda_e(1-p)$ as a function of $p$]{\label{fig:BellmanThreeThresholdCurvesB0B1}\includegraphics[width=0.48\textwidth]{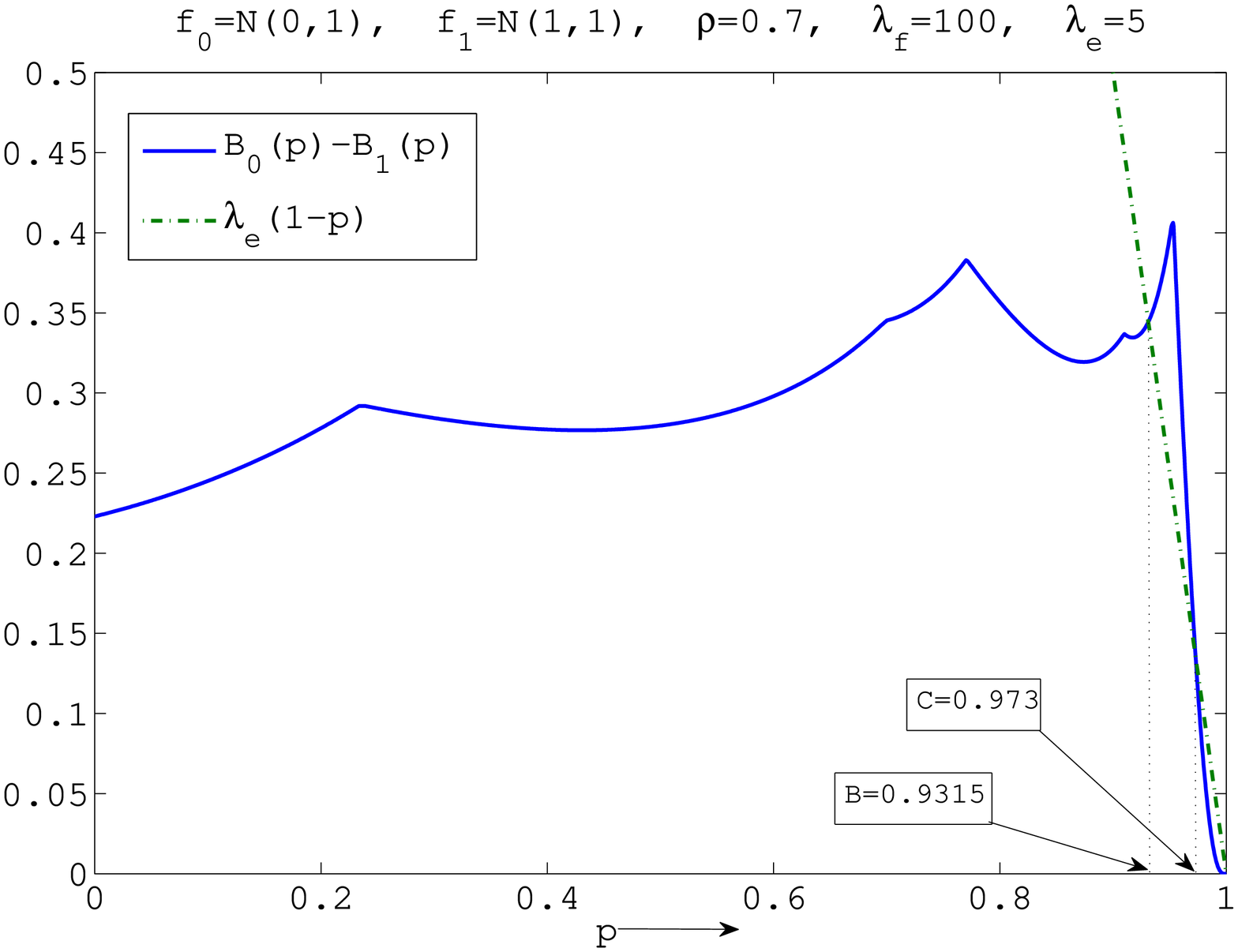}}
  \subfloat[$p+A_J(p)$ and $\lambda_f(1-p)$ as a function of $p$]{\label{fig:BellmanThreeThresholdCurvesAJ}\includegraphics[width=0.48\textwidth]{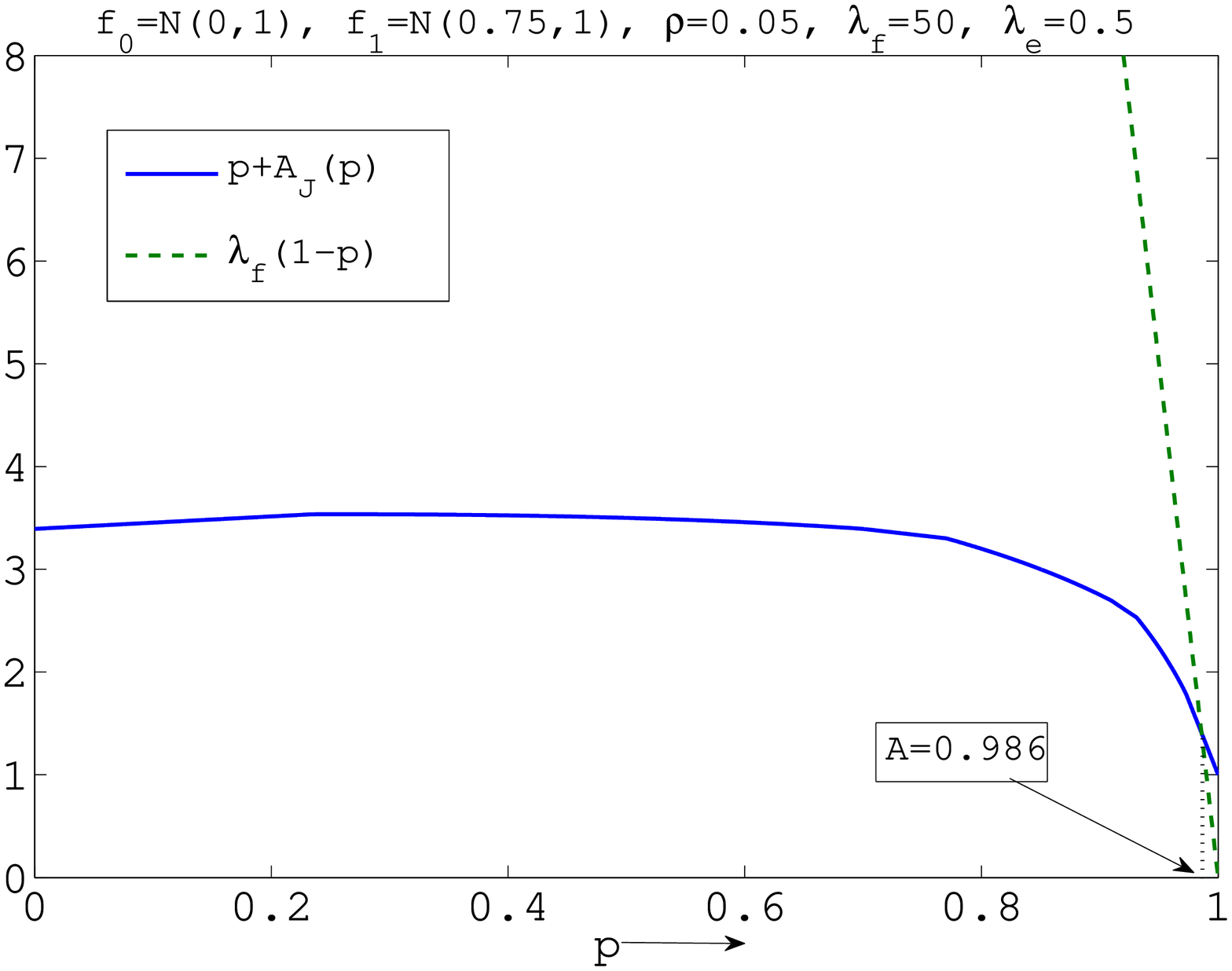}}
  \caption{Example where a two-threshold policy is not optimal: $f_0 \sim {\cal N}(0,1)$, $f_1 \sim {\cal N}(1,1)$, $\rho=0.7$, $\lambda_f=100$ and $\lambda_e=5$. Value iteration: number of iterations=1500, number of points=2000.}
  \label{fig:BellmanThreeThresholdCurves}
\end{figure}


From a practical point of view, even if a two-threshold policy or algorithm \eqref{eq:TwoThresholdAlgo} is not optimal,
one would like to use the algorithm for the following reasons.
First, as the asymptotic analysis given in Section \ref{sec:Analysis} will reveal, if the PFA constraint is moderate to small and the ANO constraint is not very severe,
then the thresholds $A$ and $B$ in $\gamma(A,B)$ can be set independently: the threshold $A$ can be set only based on the constraint $\alpha$,
and the threshold $B$ can be set based on the constraint $\beta$ alone.
Second, apart from being simple, the two-threshold algorithm \eqref{eq:TwoThresholdAlgo} is asymptotically optimal as the PFA $\to 0$.
Finally, $\gamma(A, B)$ has good trade-off curves, i.e., the ANO of $\gamma(A,B)$ can be reduced by up to 70\%, by keeping the ADD of the $\gamma(A,B)$ within 10\% of the
ADD of the Shiryaev algorithm.

It is interesting to note that a two-threshold algorithm similar to that in \eqref{eq:TwoThresholdAlgo} was shown to be exactly optimal in \cite{Girshick} for a
different but related problem of quality
control where inspection costs are considered or when the tests are destructive.

\section{Asymptotic analysis of $\gamma(A,B)$}
\label{sec:Analysis}
In this section we derive asymptotic approximations for ADD, PFA and ANO for the two-threshold algorithm $\gamma(A,B)$. To that end, we first convert the recursion for $p_k$ (see \eqref{eq:recursionSkip} and \eqref{eq:recursionTake}) to a form that
is amenable to asymptotic analysis.

Define, $Z_k = \log\frac{p_k}{1-p_k}$ for $k\geq 0$. This new variable $Z_k$ has a one-to-one mapping with $p_k$. By defining
\[a = \log\frac{A}{1-A}, \ \ \ \ b = \log\frac{B}{1-B},\]
we can write the recursions (\ref{eq:recursionSkip}) and (\ref{eq:recursionTake})
in terms of $Z_k$.

\noindent For $k\geq 1$,
\begin{equation} \label{eq:ZkrecursionTake}
Z_{k+1}  = Z_k + \log L(X_{k+1}) + |\log(1-\rho)| + \log\left(1 + \rho \, e^{-Z_k}\right), \mbox{ if } Z_k \in [b, a)
\end{equation}
and
\begin{equation} \label{eq:ZkrecursionSkip}
Z_{k+1}  =  Z_k + |\log(1-\rho)| + \log\left(1 + \rho \, e^{-Z_k}\right), \mbox{ if }Z_k \notin [b, a)
\end{equation}
with
\[Z_1 = \log\left(e^{Z_0} + \rho\right) + |\log(1-\rho)| + \log\left(L(X_1)\right) \indic_{\{Z_0 \in [b, a)\}}.\]
Here we have used the fact  that $S_{k+1}=1$ if $p_k \in [B, A)$, and $S_{k+1}=0$ otherwise (see (\ref{eq:OptimalAlgo})).
The crossing of thresholds $A$ and $B$ by $p_k$ is equivalent to the crossing of thresholds $a$ and $b$ by $Z_k$. Thus the stopping time for $\gamma(A,B)$
(equivalently $\gamma(a,b)$ with some abuse of notation) is
\[ \tau = \inf\left\{ k\geq 1: Z_k > a \right\}.\]

In this section we study the asymptotic behavior of $\gamma(a,b)$ in terms of $Z_k$,
under various limits of $a,b$ and $\rho$. Specifically, we provide two asymptotic expressions for ADD, one for fixed thresholds $a,b$, as $\rho\to 0$,
and another for fixed $b$ and $\rho$, as $a \to \infty$.
We also provide, as $a \to \infty$ and $\rho \to 0$, an asymptotic expression for PFA for fixed $b$.
Finally, we also provide asymptotic estimates of the average number of observations used before (ANO) and after the change point $\Gamma$.
Note that the limit of $a \to \infty$ corresponds to PFA $\to 0$.

Fig. \ref{fig:Zkevolution} shows a typical evolution of $\gamma(a,b)$, i.e., of $Z_k$ using (\ref{eq:ZkrecursionTake}) and
(\ref{eq:ZkrecursionSkip}), starting at time 0. Note that for $Z_k \in [b, a)$, recursion (\ref{eq:ZkrecursionTake}) is employed, while outside that interval, recursion (\ref{eq:ZkrecursionSkip}), which only uses the prior $\rho$, is employed.
As a result $Z_k$ increases monotonically outside $[b, a)$.
\begin{figure}[htb]
\center
\includegraphics[width=9cm, height=6cm]{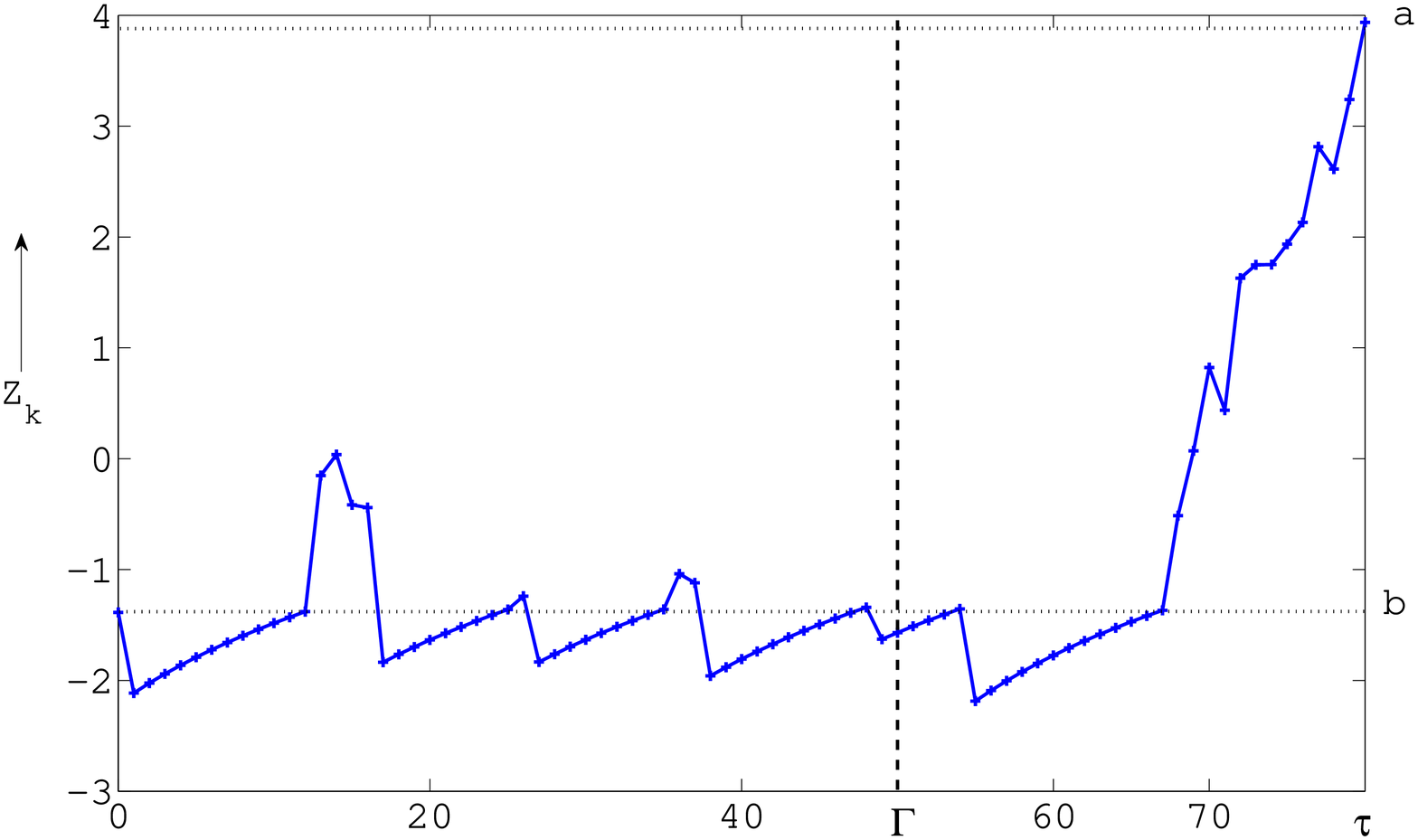}
\caption{Evolution of $Z_k$ for $f_0 \sim {\cal N}(0,1)$, $f_1 \sim {\cal N}(0.5,1)$, and $\rho=0.01$, with thresholds $a=3.89$, and $b=-1.38$, corresponding to
the $p_k$ thresholds $A=0.98$ and $B=0.2$, respectively. Also $Z_0=b$.}
\label{fig:Zkevolution}
\end{figure}

From Fig. \ref{fig:Zkevolution} again, each time $Z_k$ crosses $b$ from below, it can either increase to $a$ (point $\tau$), or it can go below $b$ and approach $b$ monotonically from below, at which time it faces a similar
set of alternatives. Thus the passage to threshold $a$ possibly involves multiple cycles of the evolution of $Z_k$ below $b$. We will show in Section~\ref{sec:Delay} that after the change point $\Gamma$,
following a finite number of cycles below $b$, $Z_k$ grows up to cross $a$, and the time spent on the cycles below $b$ is insignificant as compared to $\tau-\Gamma$, as $a \to \infty$. In fact we show that, asymptotically, the time to reach $a$ is equal to the time taken by the classical Shiryaev algorithm to reach $a$.
(Note that for the classical Shiryaev algorithm
the evolution of $Z_k$ would be based on \eqref{eq:ZkrecursionTake}).

When $Z_k$ crosses $a$ from below, it does so with an overshoot. Overshoots play a significant role in the performance of many sequential algorithms (see \cite{Siegmund},
\cite{VVV2005}) and they are central to the performance of $\gamma(a,b)$ as well. In Section \ref{sec:PFAAnalysis}, we show that
PFA depends on the threshold $a$ and the overshoot $(Z_{\tau}-a)$ as $a\to \infty$, but is \textit{not} a function of the threshold $b$.

The number of observations taken during the detection process is the total time spent by $Z_k$ between $b$ and $a$. As $a \to \infty$, $Z_k$ crosses $a$ only after change
point $\Gamma$, with high probability. The total number of observations taken can thus be divided in to two parts: the part taken before $\Gamma$ (ANO), which is the fraction of time $Z_k$
is above $b$ (and hence depends only on $b$), and the part taken after $\Gamma$. 
In Section \ref{sec:Energy} we show that, asymptotically, the average number of observations taken after $\Gamma$ is approximately equal to the delay itself.

In Section \ref{sec:Approx_AND_Numerical} we provide approximations using which the asymptotic expressions can be computed and provide numerical results to demonstrate that under various scenarios, for limiting as well as moderate values of $a$, $b$, and $\rho$,
our  asymptotic expressions for ADD, PFA and ANO provide good approximations. In Section \ref{sec:OptimalityofTwoThreshold} we use the asymptotic expressions for ADD and PFA to
show asymptotic optimality of $\gamma(a,b)$.

We begin our analysis by first obtaining the asymptotic overshoot distribution for $(Z_{\tau}-a)$ using nonlinear renewal theory \cite{Siegmund,Woodroofebook82}.
As mentioned above, this will be critical to the PFA analysis. For convenience of reference, in Table \ref{Tab:Glosaary},
 we provide a glossary of important terms used in this paper.

\begin{table} [htbp]
 \centering
 \centering
 {\scriptsize
 \caption{Glossary}
\label{Tab:Glosaary}
\begin{tabular}{|ll|ll|}
\hline
\cline{1-4}
\textbf{Symbol}&\textbf{Definition/Interpretation}&\textbf{Symbol}&\textbf{Definition/Interpretation}\\
\hline
ADD&Average detection delay                                                                                &$\lambda$&Starting at $b$, first time $Z_k$ is outside $[b,a)$\\
PFA&Probability of false alarm                                                                             &$\Lambda$&Starting at $b$, first time $Z_k$ crosses $a$ \\
$\mathrm{ANO}$&Average \# observations used before change                                                &&or crosses $b$ from below\\
$\mathrm{ANO}_1$&Average \# observations used after change                                                &$\mathrm{ADD}^s$&Starting at $b$, time for $Z_k$ to reach $a$ under $\Prob_1$, when \\
$\{X_k\}$&Observation sequence                                                                             && $Z_k$ is reset to $b$ each time it crosses $b$ from below\\
$p_k$&\textit{a posteriori} probability of change                                                          &$\lambda(x)$&Starting at $x\geq b$, first time $Z_k$ is outside $[b,a)$\\
$Z_k$&$\log\frac{p_k}{1-p_k}=\sum_{i=1}^k Y_i + \eta_k$,                                                   &$\Lambda(x)$&Starting at $x\geq b$, first time $Z_k$ crosses $a$\\
$\tau$ &First time for $p_k$ to cross $A$ or                                                               &&or crosses $b$ from below\\
&first time for $Z_k$ to cross $a=\log\frac{A}{1-A}$                                                       &$\hat{\lambda}$&Starting at $b$, first time $Z_k<b$ with $a=\infty$\\
$\{\eta_k\}$&Slowly changing sequence                                                                   &$\hat{\lambda}(x)$&Starting at $x\geq b$, first time $Z_k<b$ with $a=\infty$\\
$R(x), \ \bar{r}$&Asymptotic distribution and mean of overshoot                                         &$T_b$&Time spent by $Z_k$ below $b$, after $\Gamma$, when $\tau \geq \Gamma$\\
&when $\sum_{i=1}^k Y_i$ crosses a large threshold                                               &$\tilde{\Lambda}^x$&Starting at $x\geq b$, first time $Z_k>a$, or crosses $b$ from \\
$t(x,y)$&Time for $Z_k$ to reach $y$ starting at $x$ using \eqref{eq:ZkrecursionSkip}                      &&below, or is stopped by occurrence of change\\
$\nu(x,y)$&Time for $Z_k$ to reach $y$ starting at $x$ using \eqref{eq:ZkrecursionTake}       &$\delta^x$&The fraction of time $Z_k$ is above $b$, when stopped by $\tilde{\Lambda}^x$\\
&also, time for Shiryaev algorithm to reach $y$ starting at $x$                                &$\tilde{\nu}_b$ ($\hat{\nu}_b$)&Starting at $b$, time for $Z_k$ to reach $a$, when $Z_k$ is \\
$\nu_b,\  \nu_0$&$\nu(b,a)$ and $\nu(-\infty,a)$                                                                 &&reflected at $b$ (reset to $b$ when it crosses $b$ from below)\\
\hline
\end{tabular}}
 \end{table}

In what follows, we
use $\mathrm{E}_\ell$ and
$\mathrm{P}_\ell$ to denote, respectively, the expectation and probability measure when change happens at time $\ell$.
We use $\Expect_\infty$ and $\Prob_\infty$ to denote, respectively, the expectation and probability measure when the entire sequence $\{X_n\}$ is i.i.d.
with density $f_0$.
Note that, $g(x)=o(1)$ as $x \to x_0$ is used to denote that $g(x)\to 0$ in the specified limit.

\subsection{Asymptotic overshoot}
\label{sec:Overshoot}
In this section we characterize the overshoot distribution of $Z_k$ as it crosses $a$ as $a\to \infty$. 
In analyzing the trajectory of $Z_k$, it useful to allow for arbitrary starting point $Z_0$ (shifting the time axis).  We first combine the recursions in
(\ref{eq:ZkrecursionTake}) and (\ref{eq:ZkrecursionSkip}) to get:
\begin{eqnarray*}
Z_{k+1} = Z_k + \indic_{\{Z_k \geq b\}}\log L(X_{k+1}) + |\log(1-\rho)| + \log\left(1 + e^{-Z_k}\rho\right).
\end{eqnarray*}
By defining $Y_k = \log L(X_k) + |\log (1-\rho)|$ and expanding the above recursion, we can write an expression for $Z_n$:
\begin{eqnarray}
\label{eq:ZnEqSnEtan}
Z_n &=& \sum_{k=1}^n Y_k + \log\left(e^{Z_0} + \rho\right) + \sum_{k=1}^{n-1} \log\left(1 + e^{-Z_k}\rho\right) - \sum_{k=1}^n\indic_{\{Z_k < b\}} \log L(X_{k}) \nonumber\\
&=& \sum_{k=1}^n Y_k + \eta_n.
\end{eqnarray}
Here $\eta_n$ is used to represent all terms other than the first in the equation above:
\begin{equation}
\label{eq:etan}
\eta_n= \log\left(e^{Z_0} + \rho\right)  + \sum_{k=1}^{n-1} \log\left(1 + e^{-Z_k}\rho\right) - \sum_{k=1}^n\indic_{\{Z_k < b\}} \log L(X_{k}).
\end{equation}
As defined in \cite{Siegmund}, $\eta_n$ is a \textit{slowly changing} sequence if
\begin{equation}
\label{eq:SlowlyCngCond1}
n^{-1} \max\{ |\eta_1|, \ldots, |\eta_n|\} \xrightarrow[i.p.]{n\to\infty} 0,
\end{equation}
and for every $\epsilon > 0$, there exists $n^*$ and $\delta>0$ such that for all $n\geq n^*$
\begin{equation}
\label{eq:SlowlyCngCond2}
\mathrm{P} \{\max_{1\leq k\leq n\delta} |\eta_{n+k} - \eta_{n}| > \epsilon\} < \epsilon.
\end{equation}
If indeed $\{\eta_n\}$ is a slowly changing sequence, then the distribution of $Z_{\tau}- a$, as $a\to \infty$, is equal to the asymptotic distribution of the overshoot
when the random walk $\sum_{k=1}^n Y_k$ crosses a large positive boundary. We have the following result.

\begin{theorem}
\label{thm:AsympOvershootDist}
Let $R(x)$ be the asymptotic distribution of the overshoot when the random walk $\sum_{k=1}^n Y_k$ crosses a large positive boundary under $\Prob_1$. Then
for fixed $\rho$ and $b$, under $\Prob_1$, we have the following:
\begin{enumerate}
\item $\{\eta_n\}$ is a slowly changing sequence.
\item $R(x)$ is the distribution of $Z_{\tau}- a $ as $a\to \infty$, i.e.,
\begin{equation}
\label{eq:AsympOvershootDist}
\lim_{a\to \infty} \mathrm{P}\left[Z_{\tau}- a \leq x | \tau \geq \Gamma \right] = R(x).
\end{equation}
\end{enumerate}
\end{theorem}
\begin{IEEEproof}
When $b=-\infty$, $Z_k$ evolves as in the classical Shiryaev algorithm statistic, and it is easy to see that in this case:
\begin{eqnarray*}
\eta_n &=& \left[\log\left(e^{Z_0} + \rho\right) + \sum_{k=1}^{n-1} \log\left(1 + e^{-Z_k}\rho\right)\right] \\
&=& \log\left[e^{Z_0} + \sum_{k=0}^{n-1} \rho (1-\rho)^k \prod_{i=1}^k \frac{f_0(X_i)}{f_1(X_i)}\right]. \nonumber\\
\end{eqnarray*}
It was shown in  \cite{VVV2005} that this $\{\eta_n\}$ sequence (for $b = -\infty$), with $Z_0=-\infty$, is a slowly changing sequence.
It is easy to show that $\{\eta_n\}$ is a slowly changing sequence even if $Z_0$ is a random variable. Also, if $L_Z$ is the last time $Z_k$
crosses $b$ from below, then note that, after $L_Z$, the last term $\sum_{k=1}^n\indic_{\{Z_k < b\}} \log L(X_{k})$ in \eqref{eq:etan} vanishes,
and $\eta_n$ in \eqref{eq:etan} behaves like the $\eta_n$ for $b=-\infty$. We prove the theorem using these observations. The detailed proof is given in the appendix to this
section.
\end{IEEEproof}

\subsection{$\mathrm{PFA}$ Analysis}
\label{sec:PFAAnalysis}

We first obtain an expression for PFA as a function of the overshoot when $Z_k$ crosses $a$.
\begin{lemma}
\label{lem:PFAeq}
For fixed $\rho$ and $b$,
\begin{eqnarray*}
\mathrm{PFA} &=& \mathrm{E}[1-p_{\tau}] = e^{-a} \mathrm{E}[e^{-(Z_{\tau} - a)}|  \tau \geq \Gamma](1+o(1)) \ \ \ \ \ \  \mbox{ as } a \ \to \infty.
\end{eqnarray*}
\end{lemma}
\begin{IEEEproof}See the appendix for the proof.
\end{IEEEproof}

From Lemma \ref{lem:PFAeq}, it is evident that PFA depends on the overshoot when $Z_k$ crosses $a$ as $a\to \infty$.
Since the overshoot has an asymptotic distribution (Theorem \ref{thm:AsympOvershootDist}) that depends
only on densities $f_0$, $f_1$ and prior $\rho$, and is independent of $b$, it is natural to expect that as $a\to \infty$, PFA is completely characterized by the asymptotic
distribution $R(x)$ and is not a function of the threshold $b$. This is indeed true and is established in the following theorem.

\begin{theorem}
\label{thm:PFA}
For a fixed $b$ and $\rho$,
\begin{equation}
\label{eq:PFACequalA}
\mathrm{PFA}(\gamma(a,b)) = \left(e^{-a}\int_{0}^\infty e^{-x} dR(x)\right) (1+o(1))  \mbox{ as } a \to \infty.
\end{equation}
\end{theorem}
\begin{IEEEproof}The proof is provided in the appendix.
\end{IEEEproof}

%

\subsection{Delay Analysis}
\label{sec:Delay}

The PFA for $\gamma(a,b)$ have the following bound:
\begin{equation}
\label{eq:PFAUPPERBND}
\text{PFA} = \mathrm{E}[1-p_{\tau}] \leq 1-A = \frac{1}{1+e^a} \leq e^{-a}.
\end{equation}

Using this upper bound we can show that the ADD of $\gamma(a,b)$ is given by:
\begin{eqnarray}
\label{eq:CondEdd}
\text{ADD} &=& \mathrm{E}\left[(\tau - \Gamma)^+\right] \nonumber\\
&=& \mathrm{E}[\tau-\Gamma| \tau\geq \Gamma](1 + o(1)) \  \mbox{ as } a \to \infty.
\end{eqnarray}
In the following we provide two different expressions for $\mathrm{E}[\tau-\Gamma| \tau \geq \Gamma]$. The first one is obtained by keeping $b$ fixed and taking $\rho\to 0$.
This expression will be used to get accurate delay estimates for $\gamma(a,b)$ in Section \ref{sec:Approx_AND_Numerical}

Next, we will provide another asymptotic expression for $\mathrm{E}[\tau-\Gamma| \tau \geq \Gamma]$ for a fixed $b$, $\rho$ and as $a\to \infty$. We show that
in this limit, $\mathrm{E}[\tau-\Gamma| \tau \geq \Gamma]$ converges to the Shiryaev delay. This fact will be used to prove the asymptotic optimality of $\gamma(a,b)$
in Section \ref{sec:OptimalityofTwoThreshold}.

It was discussed in reference to Fig. \ref{fig:Zkevolution} that  each time $Z_k$ crosses $b$ from below, it faces two alternatives, to cross $a$
without ever coming back to $b$ or to go below $b$ and cross it again from below. It was mentioned that the passage to the threshold $a$ is through
multiple such cycles. Motivated by this we define the following stopping times $\lambda$ and $\Lambda$:
\begin{equation}
\label{eq:TwoSidedShiryaev}
\lambda  \defeq \inf\{k\geq 1: Z_k \notin [b, a), Z_0=b \},
\end{equation}
and
\begin{equation}
\label{eq:ExtendedTest}
\Lambda  \defeq \inf\{k\geq 1: Z_k >a \ \mbox{ or } \exists \ k \ s.t. \ Z_{k-1}<b \mbox{ and } Z_k\geq b\ , Z_0=b \}.
\end{equation}
Let $t(x,y)$ be the constant time taken by $Z_k$ to move from $Z_0=x$ to $y$ using the recursion (\ref{eq:ZkrecursionSkip}), i.e.
\begin{equation}
\label{eq:tXY}
t(x,y) \defeq \inf\{k\geq 0: Z_k > y, Z_0=x, ~ x,y \notin [b,a)\}.
\end{equation}
Then, we can write $\Lambda$ as a function of $\lambda$ using \eqref{eq:tXY}:
\begin{equation*}
\Lambda = (\lambda + t(Z_\lambda,b))\indic_{\{Z_\lambda < b\}} + \lambda \, \indic_{\{Z_\lambda > a\}} = \lambda + t(Z_\lambda,b) \indic_{\{Z_\lambda < b\}}.
\end{equation*}
The significance of these stopping times is as follows.
If we start the process at $Z_0=b$ and \textit{reset $Z_k$ to $b$ each time it crosses $b$ from below},
then the time taken by $Z_k$ to move from $b$ to $a$ is the sum of a finite but random number of random variables with distribution of $\Lambda$,
say $\Lambda_1, \Lambda_2, \ldots, \Lambda_N$. For $i =1, \ldots, N-1$, $Z_{\Lambda_i} < b$, and $Z_{\Lambda_N} > a$. Thus the time
to reach $a$ in this case is $\Expect_1\left[ \sum_{k=1}^{N} \Lambda_k\right]$.
Let
\[\mathrm{ADD}^s \defeq \Expect_1\left[ \sum_{k=1}^{N} \Lambda_k\right].\]

The behavior of the delay path depends on $Z_\Gamma$, the value of $Z_k$ at the change point $\Gamma$, and how $Z_k$ evolves after that point.
We use $\{Z_k \nearrow b\}$ to indicate that $Z_k$ approaches $b$ from below for some $k > \Gamma$, i.e. $\exists k > \Gamma, s.t., Z_{k-1}<b, Z_k\geq b$. and use $\{Z_k \nearrow a\}$ to represent the event that $Z_k$
crossed $a$ without ever coming back to $b$, i.e., $Z_k\geq b, \forall k > \Gamma$.
 We define the following three
disjoint events:
\begin{eqnarray*}
\mathcal{A} &=& \{Z_\Gamma<b\},\\
\mathcal{B} &=& \{Z_\Gamma\geq b; Z_k \nearrow b\},\\
\mathcal{C} &=& \{Z_\Gamma\geq b; Z_k \nearrow a\}.
\end{eqnarray*}
Thus, under the event $\mathcal{A}$, the process $Z_k$ starts below $b$ at $\Gamma$, and reaches $a$ after multiple up-crossings of the threshold $b$.
Under the event $\mathcal{B}$, the process $Z_k$ starts above $b$ at $\Gamma$, and crosses $b$ before $a$. It then has multiple
up-crossings of $b$, similar to the case of event $\mathcal{A}$.
Under event $\mathcal{C}$, the process $Z_k$ starts above $b$ at $\Gamma$, and reaches $a$ without ever coming below $b$.

Also define,
\begin{equation}
\label{eq:Lambdax}
\lambda(x) = \inf\{k\geq 1: Z_k \notin [b, a), Z_0=x, b \leq x < a\},
\end{equation}
and let $\Lambda(x)$ be defined with $Z_0=x$ similar to (\ref{eq:ExtendedTest}). Thus, $\lambda$ and $\lambda(b)$ have the same distribution. Similarly, $\Lambda$ and $\Lambda(b)$ are identically distributed.

The following theorem gives an asymptotic expression for the conditional delay.
\begin{theorem}
\label{thm:ADDExact}
For a fixed values of the thresholds $a,b$, the conditional delay is given by
\begin{eqnarray}
\label{eq:CondDelayA}
\hspace{-2cm}\mathrm{E}[\tau-\Gamma| \tau \geq \Gamma] \hspace{0.2cm}
  =&& \hspace{-0.5cm}\bigg[ \mathrm{ADD}^s \ \ \mathrm{P}(\mathcal{A} \cup \mathcal{B}| \tau \geq \Gamma) \bigg.\nonumber \\
  &+& \mathrm{E}[\Lambda(Z_\Gamma) | \mathcal{C}, \tau \geq \Gamma] \ \ \mathrm{P}(\mathcal{C}| \tau \geq \Gamma) \nonumber \\
  &+& \mathrm{E}[t(Z_\Gamma,b)| \mathcal{A}, \tau \geq \Gamma] \ \ \mathrm{P}(\mathcal{A}| \tau \geq \Gamma) \nonumber \\
  &+& \bigg.\mathrm{E}[\Lambda(Z_\Gamma) | \mathcal{B}, \tau \geq \Gamma] \ \ \mathrm{P}(\mathcal{B}| \tau \geq \Gamma) \bigg] \big( 1+ o(1)\big) \mbox{ as } \rho \to 0.
\end{eqnarray}
\end{theorem}
\begin{IEEEproof}
The proof is provided in the appendix.
\end{IEEEproof}
In Section \ref{sec:Approx_AND_Numerical} we will provide approximations for various terms in \eqref{eq:CondDelayA} to get an accurate estimate of ADD.
In Lemma \ref{eq:ADDsENuEqui} we provide expressions for $\mathrm{ADD}^s$.

Let $\Psi$ represent the Shiryaev recursion, i.e., updating $Z_k$ using only (\ref{eq:ZkrecursionTake}). Define
\begin{equation}
\label{eq:nuxy}
\nu(x,y) = \inf\left\{ k\geq 1: \Psi(Z_{k-1}) > y, \ \ Z_0=x\right\}.
\end{equation}
Thus, $\nu(x,y)$ is the time for the Shiryaev algorithm to reach $y$ starting at $x$.
Also, define the stopping times:
\begin{equation}
\label{eq:ShiryaevNuBtoC}
\nu_b = \nu(b,a),
\end{equation}
and
\begin{equation}
\label{eq:ShiryaevNuBtoCStartAt0}
\ \nu_0 = \nu(-\infty,a).
\end{equation}
Note that, $\nu_0$ is the stopping time for the classical Shiryaev algorithm \cite{Shiryaev63} and $\nu_b$ is its modified form which starts at $b$.
We have the following asymptotic expression.
\begin{lemma}
\label{eq:ADDsENuEqui}
For a fixed $b$ and $\rho$, $\mathrm{ADD}^s$, the average time for $Z_k$ to cross $a$ starting at $b$, under $\mathrm{P}_1$, with $Z_k$ reset to $b$ each time it crosses $b$
from below, is given by
\begin{equation}
\mathrm{ADD}^s = \frac{\mathrm{E}_1[\lambda] + \mathrm{E}_1[t(Z_\lambda, b) | \{Z_\lambda < b\}] \mathrm{P}_1(Z_\lambda < b)}{\mathrm{P}_1(Z_\lambda > a)},
\end{equation}
and is asymptotically equal to the time taken by the Shiryaev algorithm to move from $b$ to $a$, i.e.,
\begin{equation}\label{eq:ADDs}
\mathrm{ADD}^s = \mathrm{E}_1 [\nu_b] (1 + o(1)) \mbox{ as } a \to \infty.
\end{equation}
\end{lemma}
\begin{IEEEproof} We have
\begin{eqnarray*}
\mathrm{ADD}^s &=& \mathrm{E}_1\left[ \sum_{k=1}^{N} \Lambda_k\right] \\
&\overset{({\romannumeral 1})}{=}& \mathrm{E}_1[N] \mathrm{E}_1[\Lambda]  \nonumber\\
       &\overset{({\romannumeral 2})}{=}& \frac{\mathrm{E}_1[\Lambda]}{\mathrm{P}_1(Z_\lambda > a)}\nonumber\\
      &=&  \frac{\mathrm{E}_1[\lambda] + \mathrm{E}_1[t(Z_\lambda, b) | \{Z_\lambda < b\}] \mathrm{P}_1(Z_\lambda < b)}{\mathrm{P}_1(Z_\lambda > a)}.
\end{eqnarray*}
In the above equation, equality $({\romannumeral 1})$ follows from Wald's lemma \cite{Siegmund}, and equality $({\romannumeral 2})$ follows because $N \sim \text{Geom}(P(Z_\lambda > a))$.
To obtain \eqref{eq:ADDs}, the main idea of the proof is to find stopping times which upper and lower bound the Shiryaev time on average and have delay equal to
$\frac{\mathrm{E}_1[\lambda]}{\mathrm{P}_1(Z_\lambda > a)}$ as $a\to \infty$. The details are provided in the appendix.
\end{IEEEproof}

Note that Theorem \ref{thm:ADDExact} takes $\rho\to 0$. We now provide another expression for $\mathrm{E}[\tau-\Gamma| \tau \geq \Gamma]$,
for a fixed $b$ and $\rho$ as $a\to \infty$, which will be used to prove the asymptotic optimality of $\gamma(a,b)$ in Section~\ref{sec:OptimalityofTwoThreshold}.


\begin{theorem}
\label{lem:ADD_1}
For a fixed $b$ and $\rho$, we have as $a \to \infty$
\begin{equation}
\label{eq:delayADDsUB}
\mathrm{E}[\tau-\Gamma| \tau \geq \Gamma] \leq \mathrm{ADD}^s\left(1 + o(1) \right),
\end{equation}
and hence, we have
\begin{equation}
\label{eq:ADDFirstOder}
\mathrm{E}[\tau-\Gamma| \tau \geq \Gamma] = \left[ \frac{a}{D(f_1, f_0) + |\log(1-\rho)|}\right]\big(1 + o(1)\big) \mbox{ as } a\to \infty,
\end{equation}
where, $D(f_1, f_0)$ is the K-L divergence between $f_0$ and $f_1$.
\end{theorem}
\begin{IEEEproof}
To get \eqref{eq:delayADDsUB}, we show that $\mathrm{ADD}^s$
is the dominant term in an upper bound to $\mathrm{E}[\tau-\Gamma| \tau \geq \Gamma]$ as $a \to \infty$. The steps followed are very similar to those
used to obtain \eqref{eq:CondDelayA}. The proof is given in the appendix.

To obtain \eqref{eq:ADDFirstOder}, from Lemma \ref{eq:ADDsENuEqui} and \eqref{eq:delayADDsUB} we have,
\[\mathrm{E}[\tau-\Gamma| \tau \geq \Gamma] \leq \mathrm{E}_1 [\nu_b] (1 + o(1)) \mbox{ as } a \to \infty.\]
To evaluate $\mathrm{E}_1 [\nu_b]$, following steps similar to those in Section \ref{sec:Overshoot}, it is easy to show that evolution of $Z_k$ from $b$ to $a$, with $Z_0=b$,
is according to the random walk $\sum_k \log L(X_k) + |\log (1-\rho)|$ and a slowly changing term. Thus, according to Lemma 9.1.3, pg 191 of \cite{Siegmund},
\[\mathrm{E}_1 [\nu_b] = \left[ \frac{a}{D(f_1, f_0) + |\log(1-\rho)|}\right]\big(1 + o(1)\big) \mbox{ as } a\to \infty,\]
and
\[\mathrm{E}[\tau-\Gamma| \tau \geq \Gamma] \leq \left[ \frac{a}{D(f_1, f_0) + |\log(1-\rho)|}\right]\big(1 + o(1)\big) \mbox{ as } a\to \infty.\]

To complete the proof of Theorem \ref{lem:ADD_1}, we now show that $\mathrm{E}[\tau-\Gamma| \tau \geq \Gamma]$ is asymptotically lower bounded by $\mathrm{E}_1[\nu_b]$.
From Theorem 1 in \cite{VVV2005},
\[\mathrm{E}[\nu_0-\Gamma| \nu_0 \geq \Gamma] \geq \frac{a}{D(f_1, f_0) + |\log(1-\rho)|}(1 + o(1)) \mbox{ as } a\to \infty. \]
Also, from Theorem \ref{thm:PFA},
\[\Prob[\tau < \Gamma] = \Prob[\nu_0 < \Gamma](1 + o(1)) \mbox{ as } a\to \infty.\]
Thus, we have
\[\mathrm{E}[\tau-\Gamma| \tau \geq \Gamma] \geq \mathrm{E}[\nu_0 - \Gamma| \nu_0 \geq \Gamma] (1 + o(1)) \mbox{ as } a\to \infty.\]
This is true because Shiryaev algorithm is optimal for problem \eqref{eq:basicproblem} with $\beta=\infty$.
This completes the proof. 
\end{IEEEproof}

\smallskip

\subsection{Computation of $\mathrm{ANO}$}
\label{sec:Energy}

First note that,
\begin{eqnarray*}
\mathrm{ANO} &=& \mathrm{E}\left[\sum_{k=1}^{\min\{\tau, \Gamma-1\}} S_k \right]\\
&=& \mathrm{E}\left[\sum_{k=1}^{\Gamma-1} S_k \bigg| \tau \geq \Gamma \right]\mathrm{P}(\tau \geq \Gamma) + \mathrm{E}\left[\sum_{k=1}^{\tau} S_k \bigg| \tau < \Gamma\right]\mathrm{P}(\tau < \Gamma)\\
&=&\mathrm{E}\left[\sum_{k=1}^{\Gamma-1} S_k \bigg| \tau \geq \Gamma \right]\left(1 + o(1)\right)
 \ \ \ \mbox{ as } \ \ \ a \to \infty.
 \end{eqnarray*}
 The last equality follows because $\sum_{k=1}^{\tau} S_k \leq \Gamma$ on $\{\tau < \Gamma\}$, and
 $\mathrm{P}(\tau < \Gamma) < e^{-a} \to 0$ as $a\to \infty$.

Following \eqref{eq:TwoSidedShiryaev}, we define
\begin{equation}
\label{eq:LambdaInfty}
\hat{\lambda} = \inf\{k \geq 1: Z_k < b, Z_0=b, a=\infty\}.
\end{equation}
The theorem below an gives asymptotic expression for $\mathrm{ANO}$.
\begin{theorem}
\label{lem:Energy}
For fixed $b$, we have as $a \to \infty$, and as $\rho\to 0$,
\begin{eqnarray*}
\mathrm{ANO} =\frac{ \mathrm{E}_\infty[\hat{\lambda}]}{\mathrm{P}_\infty[\Gamma \leq \hat{\lambda} + t(Z_{\hat{\lambda}},b)]} \frac{1}{1+e^b}(1 + o(1)),
\end{eqnarray*}
where, $\hat{\lambda}$ is as defined in \eqref{eq:LambdaInfty}.
\end{theorem}
\begin{IEEEproof}
Let $t(b)$ be the first time $Z_k$ crossed $b$ from below, i.e., $t(b)=t(z_0, b)$. Using the fact that observations are used only after $t(b)$, we can write the following:
\begin{eqnarray}
\label{eq:TwoTermsForANO}
\mathrm{ANO} \hspace{-0.3cm}&=&\hspace{-0.3cm} \mathrm{E}\left[\sum_{k=1}^{\Gamma-1} S_k \bigg| \tau \geq \Gamma \right] \nonumber\\
&=&\hspace{-0.3cm} \mathrm{E}\left[\sum_{k=t(b)}^{\Gamma-1} S_k \bigg| \Gamma > t(b), \tau \geq \Gamma \right] \mathrm{P}(\Gamma > t(b)| \tau \geq \Gamma ).
\end{eqnarray}
We now compute each of the two terms in \eqref{eq:TwoTermsForANO}.
For the first term in \eqref{eq:TwoTermsForANO}, we have the following lemma.
\begin{lemma}
\label{lem:ANOBeforeGamma}
For a fixed $b$, as $a\to\infty$, $\rho \to 0$,
\[
\mathrm{E}\left[\sum_{k=t(b)}^{\Gamma-1} S_k \bigg| \Gamma > t(b), \tau \geq \Gamma \right] = \frac{\ \mathrm{E}_\infty[\hat{\lambda}]}{\mathrm{P}_\infty[\Gamma \leq \hat{\lambda} + t(Z_{\hat{\lambda}},b)]}(1+o(1)).
\]
\end{lemma}
\begin{IEEEproof}Note that
\[
\lim_{a\to\infty} \mathrm{E}\left[\sum_{k=t(b)}^{\Gamma-1} S_k \bigg| \Gamma > t(b), \tau \geq \Gamma \right]
= \mathrm{E}\left[\sum_{k=t(b)}^{\Gamma-1} S_k \bigg| \Gamma > t(b), a=\infty \right].
\]
To compute the right hand side of the above equation, note that conditioned on $\{\Gamma > t(b)\}$, $\sum_{k=t(b)}^{\Gamma-1} S_k$ is
approximately the number of observations used when the process $Z_k$ starts at $Z_0=b$, goes through multiple cycles below $b$,
with each cycle length having distribution of $\hat{\lambda}$, and the sequence of cycles is interrupted by occurrence of change.
See the appendix for the detailed proof.
\end{IEEEproof}

For the second term in \eqref{eq:TwoTermsForANO}, we show that $\mathrm{P}(\Gamma > t(b)| \tau \geq \Gamma )$ is equal to $\frac{1}{1+e^b}$ in the limit and is independent of $z_0$.
\begin{lemma}
\label{lem:ANOProb}
\[\mathrm{P}(\Gamma > t(b)| \tau \geq \Gamma ) = \frac{1}{1+e^b} + o(1) \ \ \ \mbox{ as } \ \ a \to \infty, \rho \to 0. \]
\end{lemma}
\begin{IEEEproof}
The proof is provided in the appendix.
\end{IEEEproof}

The Lemmas \ref{lem:ANOBeforeGamma} and \ref{lem:ANOProb} taken together completes the proof of Theorem \ref{lem:Energy}.
\end{IEEEproof}
\smallskip

Define,
\[\mathrm{ANO}_1 = \mathrm{E}\left[\sum_{k=\Gamma}^{\tau} S_k \bigg| \tau \geq \Gamma \right].\]
Thus, $\mathrm{ANO}_1$ is the average number of observations used after the change point $\Gamma$. In some applications it
might be of interest to have an estimate of $\mathrm{ANO}_1$ as well. The following theorem shows that $\mathrm{ANO}_1$
is approximately equal to the delay itself.

\begin{theorem}
\label{lem:ANO1}
For fixed $b$ and $\rho$, we have
\begin{eqnarray*}
\mathrm{ANO}_1 = \mathrm{E}_1[\nu_b] (1 + o(1)), \ \ \mbox{ as } a \to \infty.
\end{eqnarray*}
\end{theorem}
\begin{IEEEproof}
The number of observations used after $\Gamma$ can be written as the difference between the time for $Z_k$ to reach $a$ and the time spend by it below $b$. For this we define the variable \[T_b \stackrel{\triangle}{=} \mathrm{E}\left[\sum_{k=\Gamma}^{\tau} 1_{\{Z_k<b\}}\bigg| \tau \geq \Gamma\right].\]
Thus
\[\mathrm{ANO}_1 = \mathrm{E}\left[\tau-\Gamma| \tau \geq \Gamma\right]-T_b + 1.\]
We know from Theorem \ref{lem:ADD_1} that $\mathrm{E}\left[\tau-\Gamma| \tau \geq \Gamma\right] \approx \mathrm{E}_1[\nu_b]$.
As $a\to \infty$, $T_b$ converges, and therefore $\mathrm{ANO}_1 \approx \mathrm{E}_1[\nu_b]$ for large $a$ as well. The detailed proof is given in the appendix.
\end{IEEEproof}

\section{Approximations and Numerical results}
\label{sec:Approx_AND_Numerical}
In Sections \ref{sec:PFAAnalysis}-\ref{sec:Energy}, we have obtained asymptotic expressions for ADD, PFA, and ANO as a function of the system parameters: the thresholds $a$, $b$, the densities $f_0$ and $f_1$, and the prior $\rho$. We now provide approximations for some of the analytical expressions obtained in these sections, and also provide numerical results to validate the analysis. The observations are assumed to be Gaussian with $f_0 \sim {\cal N} (0,1)$, and $f_1 \sim {\cal N} (\theta,1)$, $\theta>0$, for the simulations and analysis. In the simulations, the PFA values are computed using the expression $\Expect[1-p_\tau]$. This guarantees a faster convergence for small values of PFA.

\subsection{Numerical results for $\mathrm{PFA}$}
\label{sec:NumericalPFA}
By Theorem \ref{thm:PFA}, we have the following approximation for PFA:
\[\text{PFA} \approx  e^{-a} \int_{0}^\infty e^{-x} dR(x). \]
We note that $\int_{0}^\infty e^{-x} dR(x)$ and $\bar{r}$ can be computed numerically, at least for Gaussian observations \cite{Siegmund}.
In this section we provide numerical results to
show the accuracy of the above expression for PFA.

In Table \ref{Tab:PFAONLY} we compare the analytical approximation
with the PFA obtained using simulations of $\gamma(a,b)$ for various choices of $\rho$, thresholds $a,b$, and post change mean $\theta$.
From the table we see that the analytical approximation is quite good.
\begin{table} [htbp]
 \centering
 \centering
 {
 \caption{PFA: for $f_0 \sim {\cal N}(0,1)$, $f_1 \sim {\cal N}(\theta,1)$}
\label{Tab:PFAONLY}
\begin{tabular}{|l|l|l|l|l|l|l|l|}
\hline
        &      &   &   &$\mathrm{PFA}$        &$\mathrm{PFA     }$\\
$\theta$&$\rho$&$a$&$b$&$\mathrm{Simulations}$&$\mathrm{Analysis}$\\
\hline
0.4&0.01  &3.0  &0          &3.78$\times10^{-2}$      &3.94$\times10^{-2}$\\
0.4&0.01  &6.0  &2.0        &1.955$\times10^{-3}$     &1.96$\times10^{-3}$\\
0.75&0.01 &9.0  &-2.0       &7.968$\times10^{-5}$&7.964$\times10^{-5}$\\
2.0&0.01  &5.0  &-4.0       &2.15$\times10^{-3}$&2.155$\times10^{-3}$\\
0.75&0.005&7.6  &3.0        &3.231$\times10^{-4}$&3.235$\times10^{-4}$\\
0.75&0.1  &4.0  &-3.0       &1.143$\times10^{-2}$&1.157$\times10^{-2}$\\
\hline
\end{tabular}}
 \end{table}

In Table \ref{Tab:PFA}, we show that PFA is not a function of $b$ for large values of $a$. We fix $a=4.6$,
and increase $b$ from -2.2 to 0.85. We notice that PFA is unchanged
in simulations when $b$ is changed this way. This is also captured by the analysis and
it is quite accurate.
\begin{table} [htbp]
 \centering
 \centering
 {
 \caption{PFA $\mathrm{for}$ $\rho=0.01$, $f_0 \sim {\cal N}(0,1)$, $f_1 \sim {\cal N}(0.75,1)$}
 \label{Tab:PFA}
\begin{tabular}{|l|l|l|l|}
\hline
\cline{1-4}
 $a$&$b$&$\mathrm{Simulations}$&$\mathrm{Analysis}$\\
\hline
4.6&-2.2&6.44$\times10^{-3}$&6.48$\times10^{-3}$\\
4.6&-1.5&6.44$\times10^{-3}$&6.48$\times10^{-3}$\\
4.6&-0.85&6.44$\times10^{-3}$&6.48$\times10^{-3}$\\
4.6&0&6.44$\times10^{-3}$&6.48$\times10^{-3}$\\
4.6&0.85&6.44$\times10^{-3}$&6.48$\times10^{-3}$\\
\hline
\end{tabular}}
 \end{table}

\subsection{Approximations and numerical results for $\mathrm{ANO}$ and $\mathrm{ANO}_1$}
\label{sec:ApproxANOANO1}
We recall the expressions for $\mathrm{ANO}$ from Theorem \ref{lem:Energy} and for $\mathrm{ANO}_1$ from Theorem \ref{lem:ANO1}:
\begin{eqnarray*}
\mathrm{ANO} &\approx& \frac{ \mathrm{E}_\infty[\hat{\lambda}]}{\mathrm{P}_\infty[\Gamma \leq \hat{\lambda} + t(Z_{\hat{\lambda}},b)]} \frac{1}{1+e^b} \\
\mathrm{ANO}_1 &=& \mathrm{E}_1[\nu_b].
\end{eqnarray*}

We first simplify the expression for $\mathrm{ANO}$. Note that
\begin{eqnarray*}
\mathrm{P}_\infty[\Gamma \leq \hat{\lambda} + t(Z_{\hat{\lambda}},b)] &=& 1 - \mathrm{P}_\infty[\Gamma >  \hat{\lambda} + t(Z_{\hat{\lambda}},b)] \\
&=&  1- \mathrm{E}_\infty[(1-\rho)^{\hat{\lambda} + t(Z_{\hat{\lambda}},b)}].
\end{eqnarray*}
Thus, using Binomial approximation we get
\[\mathrm{P}_\infty[\Gamma \leq \hat{\lambda} + t(Z_{\hat{\lambda}},b)] \approx \rho\left(\mathrm{E}_\infty[\hat{\lambda}] + \mathrm{E}_\infty[t(Z_{\hat{\lambda}},b)]\right).\]
Thus, we have
\begin{eqnarray}
\label{eq:ANOApprox}
\mathrm{ANO}  \approx \frac{\rho^{-1} \ \mathrm{E}_\infty[\hat{\lambda}]}{\mathrm{E}_\infty[\hat{\lambda}] + \mathrm{E}_\infty[t(Z_{\hat{\lambda}},b)]} \frac{1}{1+e^b}.
\end{eqnarray}
We now provide approximation to compute $\mathrm{E}_\infty[\hat{\lambda}]$ and $\mathrm{E}_\infty[t(Z_{\hat{\lambda}},b)]$ in \eqref{eq:ANOApprox}.
Invoking Wald's lemma \cite{Siegmund}, we write $\mathrm{E}_\infty[\hat{\lambda}]$ as,
\[\mathrm{E}_\infty[\hat{\lambda}] = \frac{\mathrm{E}_\infty[Z_{\hat{\lambda}}] - \mathrm{E}_\infty[\eta_{\hat{\lambda}}]}{-D(f_1, f_0) + |\log(1-\rho)|}.\]
We have developed the following approximation for $\Expect_\infty[\hat{\lambda}]$:
\begin{equation}
\label{eq:ANO0_Ladder}
\Expect_\infty[\hat{\lambda}] \approx \frac{\bar{r} + \log(1+\rho e^{-b})}{D(f_1, f_0) - |\log(1-\rho)|}.
\end{equation}
Here, $\log(1+\rho e^{-b})$ is an approximation to $\mathrm{E}_\infty[\eta_{\hat{\lambda}}]$ by ignoring all the random terms after $b$ is factored out of it.
This extra $b$ will cancel with the $b$ in $\mathrm{E}_\infty[Z_{\hat{\lambda}}] = b + \mathrm{E}_\infty[Z_{\hat{\lambda}}-b]$. We approximate $\mathrm{E}_\infty[b-Z_{\hat{\lambda}}]$ by $\bar{r}$, the mean overshoot of the random walk $\sum_{i=1}^k Y_k$, with mean $D(f_1, f_0) - |\log(1-\rho)|$, when it crosses a large boundary (see (\ref{eq:ZnEqSnEtan})).

For the term  $\Expect_\infty[t(Z_{\hat{\lambda}},b)]$, we have the following lemma.
\begin{lemma}
\label{lem:txyExprn_ADD}
For fixed values of $x$ and $y$, we have
\begin{equation}
\label{eq:tXYExprn_ANO}
t(x,y) = \left(\frac{\log(1+e^y)-\log(1+e^x)}{|\log(1-\rho)|}\right)(1 + o(1)) \mbox{ as } \rho \to 0.
\end{equation}
\end{lemma}
\begin{proof}
The proof is provided in the appendix.
\end{proof}
We use \eqref{eq:tXYExprn_ANO} to get the following approximation:
\begin{equation}
\label{eq:ANO0_ZlambdatoB}
\Expect_\infty[t(Z_{\hat{\lambda}},b)] \approx \int_0^{\infty} \frac{\log(1+e^b) - \log(1+e^{b-x})}{|\log(1-\rho)|} dR(x).
\end{equation}
Thus, we approximate the distribution of $(b-Z_{\hat{\lambda}})$ by $R(x)$.

Based on the second order approximation for $\mathrm{E}_1[\nu_0]$ developed in \cite{VVV2005}, we have obtained the following approximation for $\mathrm{E}_1[\nu_b]$:
\begin{equation}
\label{eq:EddENU}
\mathrm{E}_1[\nu_b] = \frac{a - \mathrm{E}[\eta(b)] + \bar{r}}{D(f_1, f_0) + |\log(1-\rho)|} + o(1) \mbox{ as } a\to \infty,
\end{equation}
where, $\eta(b)$ is the a.s. limit of the slowly changing sequence $\eta_n$ with $Z_0=b$ under $f_1$, (see \eqref{eq:etan})
and
\begin{equation}
\label{eq:MeanOvershoot}
\bar{r} = \int_0^\infty x dR(x),
\end{equation}
with $R(x)$ as in Theorem~\ref{thm:AsympOvershootDist}.

In Table \ref{Tab:ANO0ANO1} we demonstrate the accuracy of approximations for $\mathrm{ANO}$ and $\mathrm{ANO}_1$, for
various values of $\rho$, thresholds $a,b$, and post change mean $\theta$. The table shows that the approximations are quite accurate
for the parameters chosen.
\begin{table} [htbp]
 \centering
 \centering
 {
 \caption{$f_0 \sim {\cal N}(0,1)$, $f_1 \sim {\cal N}(\theta,1)$}
\label{Tab:ANO0ANO1}
\begin{tabular}{|l|l|l|l|l|l|l|l|}
\hline
&&&&\multicolumn{2}{c|}{$\mathrm{ANO}$}&\multicolumn{2}{c|}{$\mathrm{ANO}_1$}\\
\cline{1-8}
$\theta$&$\rho$&$a$&$b$&$\mathrm{Simulations}$&$\mathrm{Analysis}$&$\mathrm{Simulations}$&$\mathrm{Analysis}$\\
\hline
0.4&0.01&8.5&-2.2&66.3&62.88&102.9&111.7\\
0.75&0.01&6.467&-2.2&34.92&34.24&27.86&29.46\\
2.0&0.01&7.5&-4.0&42.94&46.4&6.08&6.23\\
0.75&0.005&8.7&-3.0&77.18&75.09&38.73&40.38\\
0.75&0.1&8.5&0.0&2.64&3.2&21.17&22.18\\
\hline
\end{tabular}}
 \end{table}

\subsection{Approximations and numerical results for $\mathrm{ADD}$}
\label{sec:ApproxADD}
Theorem \ref{lem:ADD_1} gave a first order approximation for $\mathrm{E}[\tau-\Gamma| \tau \geq \Gamma]$:
\[\mathrm{E}[\tau-\Gamma| \tau \geq \Gamma] \approx \left[ \frac{a}{D(f_1, f_0) + |\log(1-\rho)|}\right].\]
Note that, from \cite{VVV2005}, this is also the first order approximation for the ADD of the Shiryaev algorithm,
and gives a good estimate of the delay when PFA is small. For the Shiryaev delay, a second order approximation was developed in \cite{VVV2005} (also see \eqref{eq:EddENU}):
\begin{equation*}
\mathrm{E}_1[\nu_0] = \left[ \frac{a-\Expect[\eta(-\infty)] + \bar{r}}{D(f_1, f_0) + |\log(1-\rho)|}\right] + o(1) \mbox{ as } a \to \infty.
\end{equation*}
So, instead of using $\frac{a}{D(f_1, f_0) + |\log(1-\rho)|}$, we propose to use the following:
\begin{equation}
\label{eq:ADDShiryAppr}
\mathrm{E}[\tau-\Gamma| \tau \geq \Gamma] \approx \left[ \frac{a-\Expect[\eta(-\infty)] + \bar{r}}{D(f_1, f_0) + |\log(1-\rho)|}\right].
\end{equation}

For the Shiryaev algorithm, \eqref{eq:ADDShiryAppr} provides a very good estimate of the delay even for moderate values of PFA.
In case of $\gamma(a,b)$, the accuracy of \eqref{eq:ADDShiryAppr} depends on the choice of $b$ and hence
on the constraint $\beta$, as having $b>-\infty$ increases the delay. Before we demonstrate this through numerical and simulation results we introduce the
following concept:
\begin{equation}
\mathrm{ANO}\% = \mbox{ ANO \textit{expressed as a percentage of }} \Expect[\Gamma].
\end{equation}
For example, if $\rho=0.05$, and for some choice of system parameters $\mathrm{ANO}=15$, then $\mathrm{ANO}\% = 15*0.05 = 75\%$. Thus, the concept of
$\mathrm{ANO}\%$ captures the reduction in the average number of observations used before change by employing $\gamma(a,b)$.

In Table \ref{Tab:ADDPFA} we provide various numerical examples where \eqref{eq:ADDShiryAppr} is a good approximation for
$\mathrm{E}[\tau-\Gamma| \tau \geq \Gamma]$. Since, \eqref{eq:ADDShiryAppr} is a good approximation for the Shiryaev delay as well, it
follows that, for these parameter values, the delay of $\gamma(a,b)$
is approximately equal to the Shiryaev delay. It might be intuitive that if we are aiming for
large $\mathrm{ANO}\%$ values of say 90\%, then the delay will be close to the Shiryaev delay. But values in Table \ref{Tab:ADDPFA} shows
that it is possible to achieve considerably smaller values of $\mathrm{ANO}\%$ without significantly affecting the delay.
\begin{table} [htbp]
 \centering
 \centering
 {
 \caption{$f_0 \sim {\cal N}(0,1)$, $f_1 \sim {\cal N}(\theta,1)$}
\label{Tab:ADDPFA}
\begin{tabular}{|l|l|l|l|l|l|l|l|l|}
\hline
&&&&\multicolumn{2}{c|}{ADD}&\multicolumn{2}{c|}{PFA}&\multicolumn{1}{c|}{ANO\%}\\
\cline{1-8}
$\theta$&$\rho$&$a$&$b$&$\mathrm{Simulations}$                    &$\mathrm{Analysis}$                      &$\mathrm{Simulations}$&$\mathrm{Analysis}$&\\
        &      &   &   &$\Expect[\tau-\Gamma|\tau\geq \Gamma]$&\eqref{eq:ADDShiryAppr}&                      &                   &\\
\hline
0.4&0.01&8.5&-2.2&104.9&111.7&1.608$\times10^{-4}$&1.608$\times10^{-4}$&66\%\\
0.75&0.01&6.467&-2.2&32.3&29.5&1.002$\times10^{-3}$&1.004$\times10^{-3}$&35\%\\
2.0&0.01&7.5&-4.0&6.1&6.23&1.77$\times10^{-4}$&1.768$\times10^{-4}$&43\%\\
0.75&0.005&8.7&-3.0&42.6&40.4&1.076$\times10^{-4}$&1.076$\times10^{-4}$&77\%\\
0.75&0.1&8.5&0.0&23.9&22.18&1.286$\times10^{-4}$&1.285$\times10^{-4}$&26\%\\
\hline
\end{tabular}}
 \end{table}

However, if the $\mathrm{ANO}\%$ value is small, then this means that the value of $b$ is large,
and further that the delay is large. In this case, it might happen that \eqref{eq:ADDShiryAppr} is a good approximation only for
values of PFA which are very small. This is demonstrated in Table \ref{Tab:ADDANOPercDemo}. It is clear from the table that,
for the parameter values considered, estimating the delay with less than 10\% error is only possible at PFA values of the order of $\mathrm{PFA} \approx 10^{-22}$.

\begin{table} [htbp]
 \centering
  \centering
 {
 \caption{$\rho=0.05$, $f_0 \sim {\cal N}(0,1)$, $f_1 \sim {\cal N}(0.75,1)$}
\label{Tab:ADDANOPercDemo}
\begin{tabular}{|l|l|l|l|l|l|l|}
\hline
   &   &Simulations                                &Analysis&&\\
$a$&$b$&$\Expect[\tau-\Gamma|\tau\geq \Gamma]$&\eqref{eq:ADDShiryAppr}&$\mathrm{ANO\%}$&$\mathrm{PFA}$\\
\cline{1-6}
5.0&1.0    &30   &13     &7.5\%&$4.3\times10^{-3}$\\
9.0&1.0    &42   &25     &7.5\%&$7.9\times10^{-5}$\\
13.0&1.0   &54   &37     &7.5\%&$1.4\times10^{-6}$\\
18.0&1.0   &69   &52     &7.5\%&$9.7\times10^{-9}$\\
50.0&1.0   &165  &149      &7.5\%&$1.23\times10^{-22}$\\
\hline
\end{tabular}}
 \end{table}

This motivates the need for a more accurate estimate of the delay. This is provided below.

From Theorem \ref{thm:ADDExact}, recall that we had the following three events:
\begin{eqnarray*}
\mathcal{A} &=& \{Z_\Gamma<b\},\\
\mathcal{B} &=& \{Z_\Gamma\geq b; Z_k \nearrow b\},\\
\mathcal{C} &=& \{Z_\Gamma\geq b; Z_k \nearrow a\}.
\end{eqnarray*}
As a first step towards the approximations, we ignore the event $\mathcal{B}$: $\Prob(\mathcal{B})\approx0$.
That is, we assume that if $Z_\Gamma > b$, then $Z_k$ climbs to $a$.
Define,
\[P_b = \mathrm{P}(Z_\Gamma \geq b | \tau \geq \Gamma).\]
Then \eqref{eq:CondDelayA},
\begin{equation}
\label{eq:NewADD_1}
\mathrm{E}[\tau-\Gamma| \tau \geq \Gamma] \approx P_b \ \mathrm{E}[\lambda(Z_\Gamma) | \mathcal{C}, \tau \geq \Gamma]
+ (1-P_b) (\mathrm{E}[t(Z_\Gamma,b)| \mathcal{A}, \tau \geq \Gamma] + \mathrm{ADD}^s).
\end{equation}
From Lemma \ref{eq:ADDsENuEqui}, it is easy to show the following:
\begin{eqnarray*}
\mathrm{ADD}^s
      = \mathrm{E}_1[\lambda| \{Z_\lambda > a\}] + \left(\mathrm{E}_1[\lambda| \{Z_\lambda < b\}] + \mathrm{E}_1[t(Z_\lambda, b) | \{Z_\lambda < b\}] \right) \frac{\mathrm{P}_1(Z_\lambda < b)}{1-\mathrm{P}_1(Z_\lambda < b)}.
\end{eqnarray*}
We now use the following approximations:
\begin{eqnarray*}
\mathrm{E}_1[\lambda| \{Z_\lambda > a\}]  \ &\approx& \ \mathrm{E}[\lambda(Z_\Gamma) | \mathcal{C}, \tau \geq \Gamma]
\ \approx \ \frac{a-\Expect[\eta(-\infty)] + \bar{r}}{D(f_1, f_0) + |\log(1-\rho)|},\\
\mathrm{E}_1[\lambda| \{Z_\lambda < b\}] &\approx& \frac{\bar{r} + \log(1+\rho e^{-b})}{D(f_1, f_0) - |\log(1-\rho)|},\\
\mathrm{E}_1[t(Z_\lambda, b) | \{Z_\lambda < b\}] &\approx& t(b-\bar{r}, b) \approx \frac{\log(1+e^b) - \log(1+e^{b-\bar{r}})}{|\log(1-\rho)|}.
\end{eqnarray*}
To compute \eqref{eq:NewADD_1}, we also need approximations for $\mathrm{P}_1(Z_\lambda < b)$, $P_b$ and $\mathrm{E}[t(Z_\Gamma,b)| \mathcal{A}]$.
Those are provided below.
Setting $a=\infty$ we have, by Wald's likelihood identity, Proposition 2.24, Pg 13, \cite{Siegmund},
\begin{equation*}
\mathrm{P}_1(Z_\lambda < b) = \mathrm{E}_\infty\left[\frac{f_1(X_1)\ldots f_1(X_\lambda)}{f_0(X_1)\ldots f_0(X_\lambda)}\right].
\end{equation*}
Under $\Prob_\infty$, $\lambda$ a.s. ends in $b$, and with high probability it takes very small values. Hence, this expressions can be computed using Monte Carlo simulations.
Further,
\begin{eqnarray*}
P_b &=&  \mathrm{P}(\Gamma > t(-\infty, b)) \mathrm{P}(Z_\Gamma > b | \Gamma > t(-\infty, b), \tau \geq \Gamma)\\
    &\approx& \frac{1}{1+e^b} \frac{\Expect_\infty[\hat{\lambda}]}{\mathrm{E}_\infty[\hat{\lambda}] + \mathrm{E}_\infty[t(Z_{\hat{\lambda}},b)]}.
\end{eqnarray*}
We already have the approximations for $\Expect_\infty[\hat{\lambda}]$ and $\mathrm{E}_\infty[t(Z_{\hat{\lambda}},b)]$ from Section \ref{sec:ApproxANOANO1}.
The approximation for $\mathrm{E}[t(Z_\Gamma,b)| \mathcal{A}]$ can be obtained as follows (all expectations conditioned on $\{\tau \geq \Gamma\}$):
\begin{eqnarray*}
(1-P_b)\mathrm{E}[t(Z_\Gamma,b)| \mathcal{A}] &=& (1-P_b)\mathrm{E}[t(Z_\Gamma,b)|\{Z_\Gamma<b\}]\\
&=&  \mathrm{E}[t(Z_\Gamma,b)|\{Z_\Gamma<b\} \cap \{\Gamma > t(-\infty, b)\}] \mathrm{P}(\{\Gamma > t(-\infty, b)\} \cap \{Z_\Gamma<b\}) \\
&&+ \mathrm{E}[t(Z_\Gamma,b)|\{Z_\Gamma<b\} \cap\{\Gamma \leq t(-\infty, b)\}]\mathrm{P}(\{\Gamma \leq t(-\infty, b)\} \cap \{Z_\Gamma<b\}).
\end{eqnarray*}
This can be computed using
\[ \mathrm{P}(\{\Gamma > t(-\infty, b)\} \cap \{Z_\Gamma<b\}) \approx \frac{1}{1+e^b}  \frac{\mathrm{E}_\infty[t(Z_{\hat{\lambda}},b)]}{\mathrm{E}_\infty[\hat{\lambda}] + \mathrm{E}_\infty[t(Z_{\hat{\lambda}},b)]},\] and
\[\mathrm{P}(\{\Gamma \leq t(-\infty, b)\} \cap \{Z_\Gamma<b\}) = \mathrm{P}(\{\Gamma \leq t(-\infty, b)\}) \approx \frac{e^b}{1+e^b}.\]
To compute conditional expectation of $t(Z_\Gamma,b)$, we need to subtract from $t(x,b)$, the mean of $\Gamma$ conditioned on $\{\Gamma \leq t(x,b)\}$. Specifically,
\[\mathrm{E}[t(Z_\Gamma,b)|\{Z_\Gamma<b\} \cap \{\Gamma > t(-\infty, b)\}] = t(b-\bar{r},b) - \frac{1}{\mathrm{P}(\Gamma \leq t(b-\bar{r},b))} \sum_{k=1}^{t(b-\bar{r},b)} k (1-\rho)^{k-1} \rho,\]
and,
\[\mathrm{E}[t(Z_\Gamma,b)|\{Z_\Gamma<b\} \cap\{\Gamma \leq t(-\infty, b)\}] = t(-\infty,b) - \frac{1}{\mathrm{P}(\Gamma \leq t(-\infty,b))} \sum_{k=1}^{t(-\infty,b)} k (1-\rho)^{k-1} \rho.\]
Thus we have obtained approximations for all the terms for the new approximation for $\Expect[\tau-\Gamma|\tau \geq \Gamma]$ in (\ref{eq:NewADD_1}).

In Table \ref{Tab:ADDANOPercDemo_NewADD}, we now reproduce Table \ref{Tab:ADDANOPercDemo} with a new column containing delay estimates computed using the new ADD (for $\Expect[\tau-\Gamma|\tau \geq \Gamma]$) approximation \eqref{eq:NewADD_1}. The values shows that all estimates are nearly within 10\% of the actual value.

In Table \ref{Tab:NewADDApprox}, we show the accuracy of the new ADD approximation \eqref{eq:NewADD_1}, for various values of the system parameters,
by comparing it with simulations and also with \eqref{eq:ADDShiryAppr}.
We also set PFA around $1\times 10^{-3}$.
The table clearly demonstrates that the new ADD approximation predicts the ADD with less than 10\% error.

\begin{table} [htbp]
 \centering
  \centering
 {
 \caption{$\rho=0.05$, $f_0 \sim {\cal N}(0,1)$, $f_1 \sim {\cal N}(0.75,1)$}
\label{Tab:ADDANOPercDemo_NewADD}
\begin{tabular}{|l|l|l|l|l|l|l|}
\hline
   &   &Simulations                                &Analysis&New Analysis&&\\
$a$&$b$&$\Expect[\tau-\Gamma|\tau\geq \Gamma]$&\eqref{eq:ADDShiryAppr}&$\mathrm{ADD}$ from \eqref{eq:NewADD_1}&$\mathrm{ANO\%}$&$\mathrm{PFA}$\\
\cline{1-7}
5.0&1.0    &30    &13    &34     &7.5\%&$4.3\times10^{-3}$\\
9.0&1.0    &42    &25    &46     &7.5\%&$7.9\times10^{-5}$\\
13.0&1.0   &54    &37    &58     &7.5\%&$1.4\times10^{-6}$\\
18.0&1.0   &69    &52    &73     &7.5\%&$9.7\times10^{-9}$\\
50.0&1.0   &165   &149   &169     &7.5\%&$1.23\times10^{-22}$\\
\hline
\end{tabular}}
 \end{table}

\begin{table} [htbp]
 \centering
 \centering
 {
 \caption{$f_0 \sim {\cal N}(0,1)$, $f_1 \sim {\cal N}(0.75,1)$, PFA $\approx 10^{-3}$, ANO=10\% of Shiryaev ANO}
 \label{Tab:NewADDApprox}
\begin{tabular}{|l|l|l|l|l|l|l|l|l|l|l|l|l|}
\hline
&&&\multicolumn{3}{c|}{$\mathrm{ADD}$}&\\
\cline{1-6}
$\rho$ & $a$ & $b$ & Simulations &      Analysis              &   Analysis      &    \\
       &     &     &             &  New \eqref{eq:NewADD_1}   &  \eqref{eq:ADDShiryAppr} &$\mathrm{ANO}\%$\\
\hline
0.01 & 6.4 & 2.7 & 250&260&14.42&0.33\%\\
0.005&6.45&0.6&181&190&22.09&1.5\%\\
0.001&6.47&-2.7&75&80&33.68&7.6\%\\
0.0005&6.47&-3.49&74&79&36.49&8.4\%\\
0.0001&6.47&-5.2&76&80&42.56&9.6\%\\
\hline
\end{tabular}}
 \end{table}

\section{Asymptotic Optimality and performance of $\gamma(a,b)$}
\label{sec:OptimalityofTwoThreshold}

\subsection{Asymptotic Optimality of $\gamma(a,b)$}
\label{sec:asympOptofgammaab}
In Theorem \ref{lem:ADD_1} we saw that for a fixed $b$ and $\rho$,
\[\mathrm{E}[\tau-\Gamma| \tau \geq \Gamma] = \left[ \frac{a}{D(f_1, f_0) + |\log(1-\rho)|}\right]\big(1 + o(1)\big) \mbox{ as } a\to \infty.\]
We recall that from \cite{VVV2005}, this is also the asymptotic delay of the Shiryaev algorithm.

Moreover, from Theorem \ref{thm:PFA}, the PFA for $\gamma(a,b)$ is
\begin{equation*}
\label{eq:PFAgammaAB}
\mathrm{PFA} = \left(e^{-a} \int_{0}^\infty e^{-x} dR(x)\right) (1+o(1)) \mbox{ as } a\to \infty.
\end{equation*}
Again from \cite{VVV2005}, this is the PFA for the Shiryaev algorithm. We thus have the following asymptotic optimality result for $\gamma(a,b)$.

\begin{theorem}
\label{thm:gammaABOpt}
With $\gamma = \{\tau, S_1, \ldots, S_\tau\}$ define
\[\Delta(\alpha,\beta) = \{\gamma: \mathrm{PFA}(\gamma)\leq \alpha;\ \  \mathrm{ANO}(\gamma)\leq \beta\},\]
then for a fixed $\beta$ and $\rho$,
\begin{equation}
\label{eq:asympOpt}
\mathrm{ADD}(\gamma(a(\alpha,\beta), b(\alpha,\beta))) = \left[\inf_{\gamma \in \Delta(\alpha,\beta)}  \mathrm{ADD}(\gamma)\right](1+o(1)) \mbox{ as } \alpha\to 0.
\end{equation}
Here, for each $\alpha, \beta$, $b(\alpha,\beta)$ is the smallest $b$ such that
$\mathrm{ANO}(\gamma(a(\alpha,\beta), b(\alpha,\beta)))\leq \beta$ as $a\to \infty$.
\end{theorem}
\begin{IEEEproof}
Fix $b$ such that $\mathrm{ANO}(\gamma(a,b)) \leq \beta$ as $a\to \infty$. It may happen that the constraint $\beta$ is not met with equality. Then we choose
the smallest $b$ which satisfies the constraint $\beta$ as $a\to \infty$. This choice of threshold $b$ is unique for a given $\beta$ because
ANO is not a function of threshold $a$ as $a\to \infty$.

As $a\to \infty$, the PFA and ADD both approach the Shiryaev PFA \eqref{eq:PFACequalA}
and Shiryaev delay \eqref{eq:ADDFirstOder}, respectively. Thus, as $a \to \infty$, $\gamma(a,b)$ is optimal over the class of all control policies $\Delta(\alpha,\beta)$
that satisfy the constraints $\alpha$ and $\beta$.
\end{IEEEproof}

\subsection{Trade-off curves: Performance of $\gamma(a,b)$ for a fixed and moderate $\alpha$}
\label{sec:DiscModAlpha}
Theorem \ref{thm:gammaABOpt} shows that for small values of PFA, $\gamma(a,b)$ is approximately optimal, i.e., it is not possible
to outperform $\gamma(a,b)$ by a significant margin.
But for moderate values of PFA, it is not clear if their exists algorithms which can significantly outperform $\gamma(a,b)$.
Our aim is to partially address this issue in this section.

In Fig. \ref{fig:TradeoffGauss} we plot the ANO-ADD trade-off for the two-threshold algorithm. Specifically, we compare
the two-threshold algorithm with the classical Shiryaev algorithm and study how much ANO can be reduced without significantly loosing in terms of ADD.
For Fig. \ref{fig:TradeoffGauss} we pick four values of $\rho: 0.05, 0.01, 0.005, 0.001$. For a fixed $\rho$, we fix $b=-\infty$ and select threshold
$a$ such that the $\mathrm{PFA}(\gamma(a,b))= 10^{-4}$. We then increase the threshold $b$ to have ANO\% values of $75\%, 50\%, 30\%, 15\%$.
We note that it was possible to reduce the ANO to 15\% of $\Expect[\Gamma]$ by increasing the threshold $b$ this way, without affecting
the probability of false alarm. Fig. \ref{fig:TradeoffGauss} shows that we can reduce ANO by up to 25\% while getting approximately the same ADD performance
as that of the Shiryaev algorithm. Moreover, if we allow for a 10\% increase in ADD compared to that of the Shiryaev algorithm,
then we can reduce ANO by up to 70\% (see plot for ANO\% =30\%). 
\begin{figure}[htb]
\center
\includegraphics[width=11cm, height=7cm]{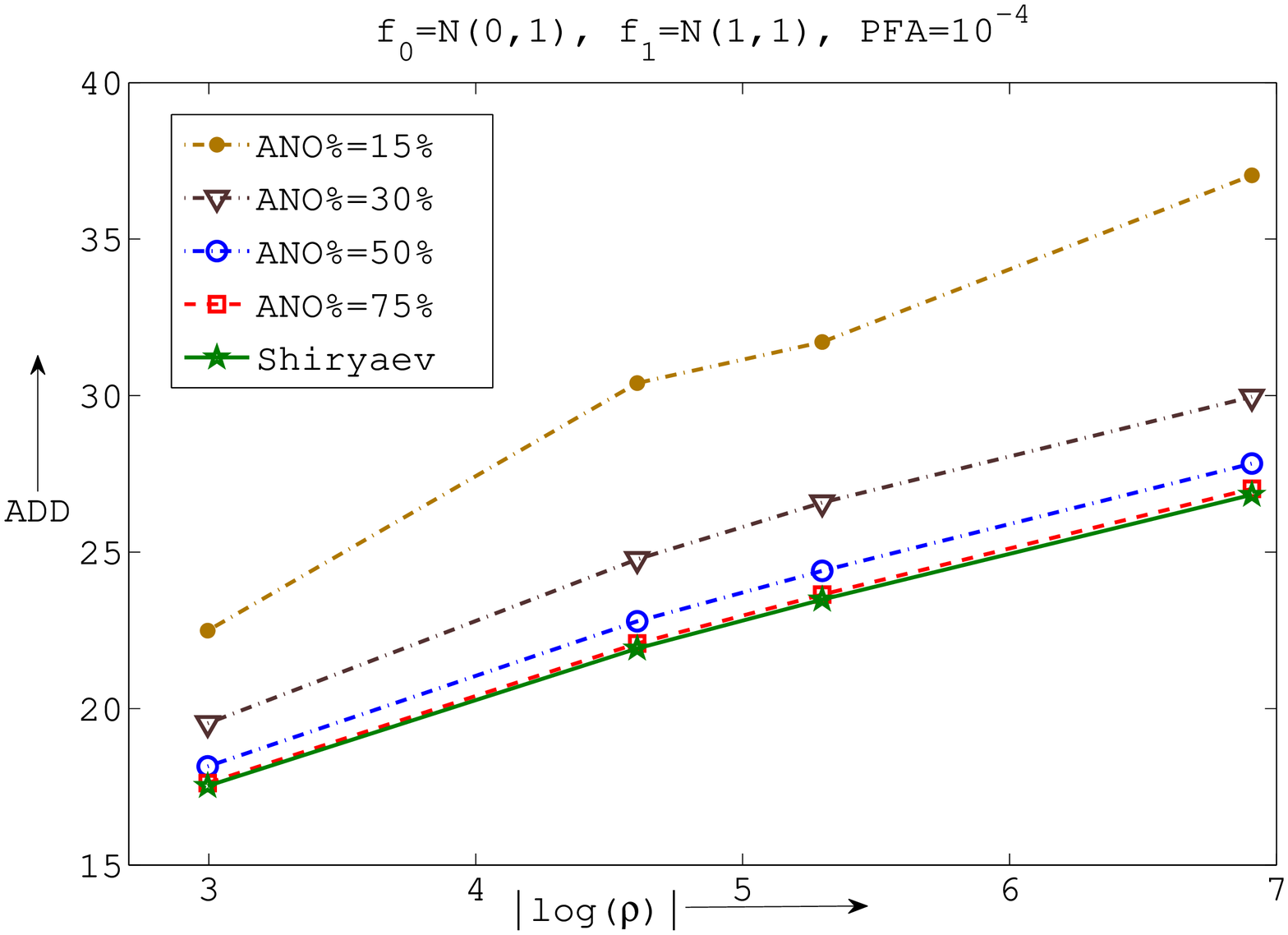}
\caption{Trade-off curves comparing performance of two-threshold algorithm with the Shiryaev algorithm for ANO\% of 75, 50, 30 and 15\%. $f_0 \sim {\cal N}(0,1)$, $f_1 \sim {\cal N}(1,1)$, and $\mathrm{PFA}=10^{-4}$.}
\label{fig:TradeoffGauss}
\end{figure}

Such a behavior was also observed in Table \ref{Tab:ADDPFA}, where we saw that the delay for $\gamma(a,b)$ 
is approximately equal to the Shiryaev delay for moderate to large ANO\% values. Thus, for moderate PFA values, when the ANO\% is moderate to large,
$\gamma(a,b)$ is approximately optimal. 

\subsection{Comparison with fractional sampling}
\label{sec:FracSampleComp}

In this section we compare the performance of $\gamma(a,b)$ with the naive approach of fractional sampling, in which an ANO\% of $\epsilon$\% is achieved by employing Shiryaev algorithm and using a sample with probability $\epsilon$. Also, in fractional sampling, when a sample is skipped, the posterior probability $p_k$ is updated using \eqref{eq:recursionSkip}.
Figure \ref{fig:CompareWithFracSample50Perc} compares the two schemes for ANO\% of 50\%. 
We also plot the performance of the Shiryaev algorithm for the same values of PFA and $\rho$.
The figure shows that $\gamma(a,b)$ helps in reducing the observation cost by a significant margin as compared to the fractional sampling scheme.

From our approximations, we know that for large $a$
\[\mathrm{ADD(\gamma(a,b))} \approx \frac{a}{D(f_1, f_0) + |\log(1-\rho)|}.\]
When the K-L distance $D(f_1, f_0)$ dominates the sum $D(f_1, f_0) + |\log(1-\rho)|$, then we would expect that any scheme that ignores the past observations
for observation control will perform poorly as compared to the one that relies on the state of the system to decide whether or not to take a sample in the next time slot.
This is verified by the figure: as $\rho\to 0$,
we see a significant difference in performances of $\gamma(a,b)$ and the fractional sampling scheme.
The figure also shows that as $\rho$ becomes large, and begins to dominate the sum $D(f_1, f_0) + |\log(1-\rho)|$, the ADD performance of the
fractional sampling scheme approach that of the two-threshold algorithm $\gamma(a,b)$.
\begin{figure}[htb]
\center
\includegraphics[width=11cm, height=7cm]{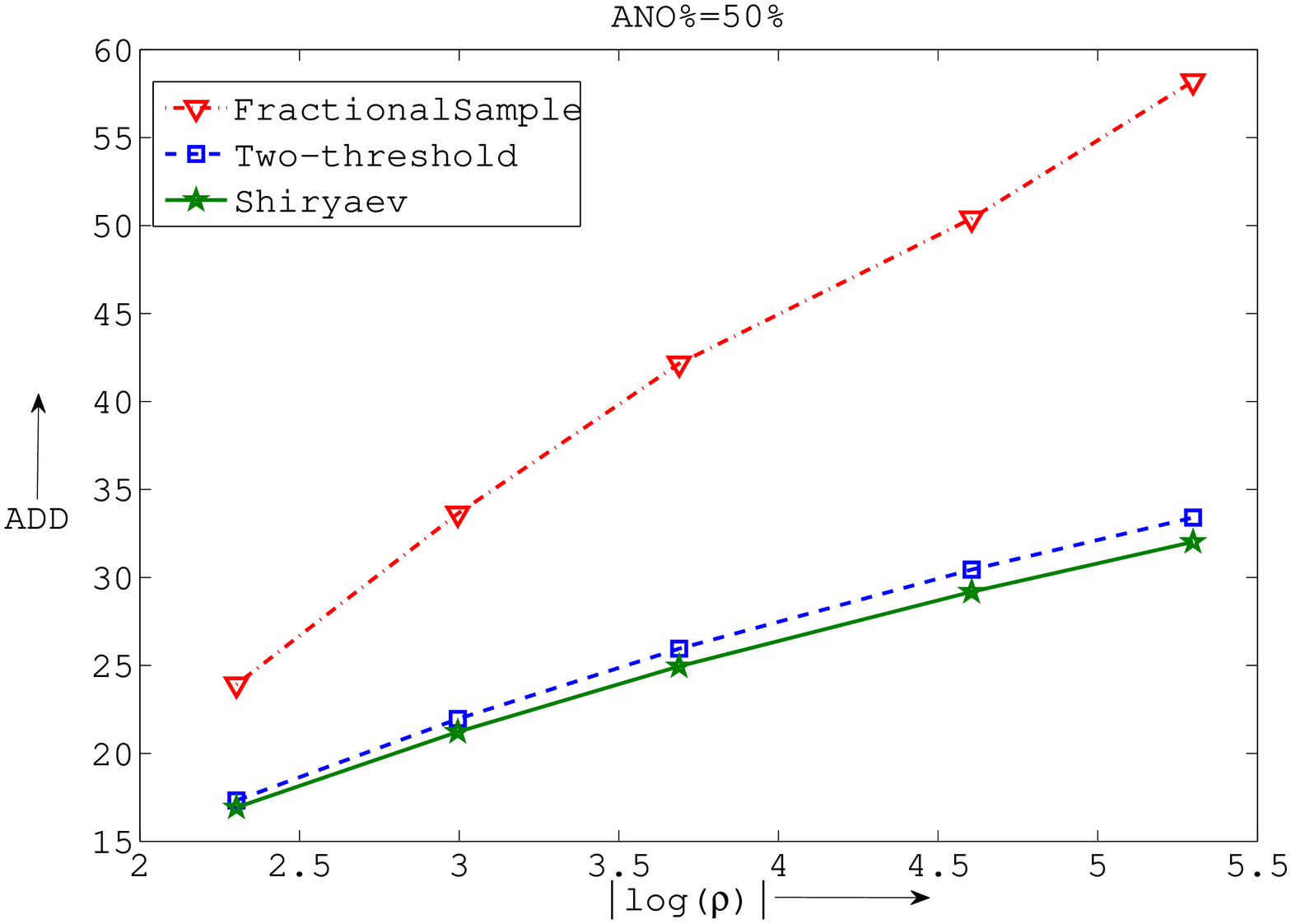}
\caption{Trade-off curves comparing performance of the two-threshold algorithm with the Fractional Sampling Scheme for ANO\% 50\%.  $f_0 \sim {\cal N}(0,1)$, $f_1 \sim {\cal N}(0.75,1)$, and $\mathrm{PFA}=10^{-3}$.}
\label{fig:CompareWithFracSample50Perc}
\end{figure}
%

\section{Conclusions}
\label{sec:Conclusion}
We posed a data-efficient version of the classical Bayesian quickest change detection problem, where we control the number of observations taken before the change occurs.
We obtained a two-threshold Bayesian algorithm that is asymptotically optimal, has good trade-off curves and is easy to design.
We derived analytical approximations for the ADD, PFA and ANO performance of the two-threshold algorithm using which we can design the algorithm by choosing the thresholds.
In particular, we showed that, 
when the constraint on the PFA is moderate to small and that on the ANO is not very small, the two-thresholds can be set independent of each other.
We also provided extensive numerical and simulation results that validate our analysis. Our results indicate that our two-threshold algorithm can significantly save on the number of observations taken before the change, while maintaining the delay relatively unchanged. A comparison with the naive approach of fractional sampling shows that the two-threshold algorithm is indeed very efficient in using observations to detect the change. Our two-threshold algorithm has many engineering applications in settings where an abrupt change has to be detected in a process under observation, but there is a cost
associated with acquiring the data needed to make accurate decisions.

An important problem for future research is to see if two-threshold policies are optimal in
non-Bayesian (e.g., minimax) settings, where we do not have a prior on $\Gamma$. In particular,
it is of interest to understand how to update the algorithm metric in a non-Bayesian setting when we skip an observation. From an application point of view,
one can design a two-threshold algorithm based on the Shiryaev-Roberts or CUSUM approaches \cite{TartaMoustaStateOftheArt},
and use the undershoot of the metric when it goes below the threshold `$b$', to design the off times.
Furthermore, if we are able to find useful lower bounds on delay for given false alarm and ANO constraints, we may be able to use these to prove asymptotic optimality of such heuristic algorithms, as is done for the standard quickest change detection problem \cite{VVV2005}, \cite{Lai}. Also, such lower bounds
can possibly help in obtaining insights for cases where the observations are not i.i.d. \cite{VVV2005}, \cite{Lai}.
Other interesting problems in this area include the design of data-efficient optimal algorithms for robust change detection and nonparametric change detection.

\section*{Appendix to Section \ref{sec:Overshoot}}

\begin{IEEEproof}[Proof of Theorem \ref{thm:AsympOvershootDist}]
We first show that $\eta_n$ with $b=-\infty$, and $Z_0$ a random variable, is a slowly changing sequence.
Let $Z_0$ takes value $z_0$, then
\begin{eqnarray*}
\eta_n = \log\left[e^{z_0} + \sum_{k=0}^{n-1} \rho (1-\rho)^k \prod_{i=1}^k \frac{f_0(X_i)}{f_1(X_i)}\right] \xrightarrow[n\to \infty]{\mathrm{P}_1- a.s.} \log\left[e^{z_0} + \sum_{k=0}^{\infty} \rho (1-\rho)^k \prod_{i=1}^k \frac{f_0(X_i)}{f_1(X_i)}\right].
\end{eqnarray*}
Define
\[
\eta(Z_0) \stackrel{\triangle}{=}  \log\left[e^{Z_0} + \sum_{k=0}^{\infty} \rho (1-\rho)^k \prod_{i=1}^k \frac{f_0(X_i)}{f_1(X_i)}\right].
\]

Note that $\eta(Z_0)$ as a function of $Z_0$ is well defined and finite under $\Prob_1$. This is because by Jensen's inequality,
for $Z_0 = z_0$,
\begin{eqnarray*}
\mathrm{E}[\eta(z_0)] &\leq& \log \left[ e^{z_0} + \sum_{k=0}^{\infty} \rho (1-\rho)^k \mathrm{E}_1 \left(\prod_{i=1}^k \frac{f_0(X_i)}{f_1(X_i)}\right)\right]\\
&=& \log \left[ e^{z_0} + \sum_{k=0}^{\infty} \rho (1-\rho)^k\right] = \log \left( e^{z_0} + 1 \right).
\end{eqnarray*}
Thus
\begin{equation}
\label{eq:Appenetan}
\eta_n \xrightarrow[b=-\infty]{\mathrm{P}_1- a.s.} \eta(Z_0)
= \log\left(e^{Z_0} + \rho\right) + \sum_{k=1}^\infty \log\left(1 + e^{-Z_k}\rho\right).
\end{equation}
This implies $\sum_{k=1}^\infty \log\left(1 + e^{-Z_k}\rho\right)$ converges a.s. for i.i.d. $\{X_k\}$ and $b=-\infty$. This series will also converge with probability 1
if we condition on a set with positive probability. 

Let change happen at $\Gamma=l$. We set $Z_0 = Z_\Gamma = Z_l$ and assume that $\{X_k\}$, $k\geq1$ have density $f_1$, which would happen after $\Gamma$.
We first show that starting with the above $Z_0$, the sequence $\eta_n$ generated in \eqref{eq:etan} is slowly changing. 

To verify the first condition \eqref{eq:SlowlyCngCond1}, from \eqref{eq:etan} note that,
\[n^{-1} \max\{ |\eta_1|, \ldots, |\eta_n|\} \leq n^{-1} \left[|\log\left(e^{Z_0} + \rho\right)| + \sum_{k=1}^{n-1} \log\left(1 + e^{-Z_k}\rho\right) + \sum_{k=1}^n\left(|\log L(X_{k})|\right)\indic_{\{Z_k < b\}} \right].\]
Since, $Z_k \to \infty$ a.s., $\log\left(1 + e^{-Z_k}\rho\right) \to 0$, also, $\indic_{\{Z_k < b\}} \to 0$ a.s.
Thus both the sequences $\{\log\left(1 + e^{-Z_k}\rho\right)\}$ and $\{\left(|\log L(X_{k})|\right)\indic_{\{Z_k < b\}}\}$ are Cesaro summable and have Cesaro sum of zero.
Thus the term inside the square bracket above, when divided by $n$, goes to zero a.s. and hence also in probability. Thus the first condition is verified.

To verify the second condition \eqref{eq:SlowlyCngCond2}, we first obtain a bound on $|\eta_{n+k} - \eta_{n}|$.
\[|\eta_{n+k} - \eta_{n}| \leq \sum_{i=n}^{n+k-1} \log\left(1 + e^{-Z_i}\rho\right) + \sum_{i=n+1}^{n+k} \left(|\log L(X_{i})|\right)\indic_{\{Z_k < b\}}.\]
Thus,
\[\max_{1\leq k\leq n\delta} |\eta_{n+k} - \eta_{n}| \leq \sum_{i=n}^{n+n\delta-1} \log\left(1 + e^{-Z_i}\rho\right) + \sum_{i=n+1}^{n+n\delta} \left(|\log L(X_{i})|\right)\indic_{\{Z_k < b\}} \stackrel{\triangle}{=} d_n^1 + d_n^2.\]
Here, for convenience of computation, we use $ d_n^1$ and  $d_n^2$ to represent the first and second partial sums respectively. Now,
\begin{eqnarray*}
\mathrm{P} \{\max_{1\leq k\leq n\delta} |\eta_{n+k} - \eta_{n}| > \epsilon\}  \leq \mathrm{P} (d_n^1 + d_n^2 > \epsilon),
\end{eqnarray*}
and we bound the probability $\mathrm{P} (d_n^1 + d_n^2 > \epsilon)$ as follows.

On the event that $E \stackrel{\triangle}{=} \{Z_k \geq b, \forall k\geq 0\}$, $d_n^2$ is identically zero, thus for $n$ large enough,
\begin{eqnarray*}
\mathrm{P} (d_n^1 + d_n^2 > \epsilon | E) = \mathrm{P} (d_n^1 > \epsilon | E) < \epsilon.
\end{eqnarray*}
This is because $d_n^1$ behaves like a partial sum of a series of type in (\ref{eq:Appenetan}).
Since the series in (\ref{eq:Appenetan}) converges if random variables are generated i.i.d. $f_1$, it will also converge if conditioned on the event
$E$.
Thus, the partial sum $d_n^1$ converges to 0 almost surely, and hence converges to 0 in probability, i.e., $\mathrm{P} (d_n^1 > \epsilon |E) \to 0$. Select,
$n=n_1^*$ such that $\forall n > n_1^*$, $\mathrm{P} (d_n^1 > \epsilon | E) < \epsilon$.

Define
\[L_Z = \sup\{k\geq 1: Z_{k-1} < b, Z_k\geq b\},\]
with $L_Z=\infty$ if no such $k$ exists. On the event $E'$, which is the compliment of $E$, $L_Z$ is a.s. finite.
Then, by noting that $d_n^2=0$ for $L_Z < n$, we get for $n$ large enough,
\begin{eqnarray*}
\mathrm{P} (d_n^1 + d_n^2 > \epsilon | E') \stackrel{\triangle}{=} \mathrm{P}_{E'} (d_n^1 + d_n^2 > \epsilon)  &\leq& \mathrm{P}_{E'} (d_n^1 + d_n^2 > \epsilon; L_Z \geq n) +
\mathrm{P}_{E'} (d_n^1 + d_n^2 > \epsilon; L_Z < n)\\
&\leq& \mathrm{P}_{E'} (L_Z \geq n) +  \mathrm{P}_{E'} (d_n^1 + d_n^2 > \epsilon; L_Z < n)\\
&=& \mathrm{P}_{E'} (L_Z \geq n) +  \mathrm{P}_{E'} (d_n^1 > \epsilon; L_Z < n)\\
&\leq& \mathrm{P}_{E'} (L_Z \geq n) +  \mathrm{P}_{E'} (d_n^1 > \epsilon | L_Z < n)\\
&<& \epsilon/2 + \epsilon/2 = \epsilon.
\end{eqnarray*}
Since, $L_Z$ is almost surely finite, $\mathrm{P}_{E'} (L_Z \geq n) \to 0$ as $n\to \infty$. Thus we can select $n=n_2^*$ such that $\forall n > n_2^*$,
$\mathrm{P}_{E'} (L_Z \geq n) < \epsilon/2$.
For the second term, note that conditioned on $L_Z < n$, $d_n^1$ behaves like a partial sum of a series of type in (\ref{eq:Appenetan}), with $Z_0$ replaced by $Z_{L_Z}$.
Since the series in (\ref{eq:Appenetan}) converges if random variables are generated i.i.d. $f_1$ beyond $L_Z$, it will also converge if conditioned on the event
$\{L_Z < n\}$.
Thus, the partial sum $d_n^1$ converges to 0 almost surely, and hence converges to 0 in probability, i.e., $\mathrm{P}_{E'} (d_n^1 > \epsilon | L_Z < n) \to 0$. Select,
$n=n_3^*$ such that $\forall n > n_3^*$, $\mathrm{P} (d_n^1 > \epsilon | L_Z < n) < \epsilon/2$. Then $n^* = \max\{n_1^*, n_2^*, n_3^*\}$, is the desired $n^*$ and pick any $\delta>0$.
Then for $n > n^*$,
\begin{eqnarray*}
\mathrm{P} (d_n^1 + d_n^2 > \epsilon) &=& \mathrm{P} (d_n^1 + d_n^2 > \epsilon | E) \mathrm{P} (E) + \mathrm{P} (d_n^1 + d_n^2 > \epsilon | E') \mathrm{P} (E')\\
&<&  \epsilon  \mathrm{P} (E) + \epsilon   \mathrm{P} (E') <  \epsilon.
\end{eqnarray*}

Since the sequence $\eta_n$ is slowly changing, according to \cite{Siegmund}, the asymptotic distribution of the overshoot when $Z_k$ crosses a large boundary under
$f_1$ is $R(x)$.
Thus we have the following result,
\[\lim_{a\to \infty} \mathrm{P}_\ell\left[Z_{\tau}- a \leq x | \tau \geq l \right] = R(x),\]
where $P_\ell$ is the probability measure with change happening at $l$.
Now,
\[\mathrm{P}\left[Z_{\tau}- a \leq x | \tau \geq \Gamma \right] = \sum_{l=1}^\infty \mathrm{P}_l\left[Z_{\tau}- a \leq x | \tau \geq l \right] \mathrm{P}(\Gamma=l| \tau \geq \Gamma), \]
and
\[\lim_{a\to \infty} \mathrm{P}_l\left[Z_{\tau}- a \leq x | \tau \geq l \right] \mathrm{P}(\Gamma=l| \tau \geq \Gamma) = R(x) \mathrm{P}(\Gamma=l) \leq 1.\]
Hence we have the desired result by dominated convergence theorem.
\end{IEEEproof}

\section*{Appendix to Section \ref{sec:PFAAnalysis}}

\begin{IEEEproof}[Proof of Lemma \ref{lem:PFAeq}]
Since, $p_\tau > A$ imply $Z_\tau>a$, we have,
\[\frac{1}{1+e^{-Z_\tau}} \geq \frac{1}{1+e^{-a}}.\]
The required result is obtained by obtaining upper and lower bounds on PFA as follows.
\begin{eqnarray*}
\mathrm{PFA} = \mathrm{E}[1-p_{\tau}] = \mathrm{E}\left[\frac{1}{1+e^{Z_\tau}}\right] \leq \mathrm{E}\left[e^{-Z_\tau}\right].
\end{eqnarray*}
Also, 
\begin{eqnarray*}
\mathrm{PFA} = \mathrm{E}[1-p_{\tau}] = \mathrm{E}\left[\frac{1}{1+e^{Z_\tau}}\right]
&=& \mathrm{E}\left[\frac{1}{e^{Z_\tau}} \frac{1}{1+e^{-Z_\tau}}\right] \\
&\geq& \mathrm{E}\left[\frac{1}{e^{Z_\tau}} \frac{1}{1+e^{-a}}\right] = \mathrm{E}\left[e^{-Z_\tau}\right](1+o(1)) \mbox{ as } a \to \infty.
\end{eqnarray*}
Thus,
\[\mathrm{PFA} = \mathrm{E}[e^{-Z_{\tau}}](1+o(1)) = e^{-a} \mathrm{E}[e^{-(Z_{\tau} - a)}](1+o(1)) \mbox{ as } a \ \to \infty.\]
Now note that,
\[\mathrm{E}[e^{-(Z_{\tau} - a)}] = \mathrm{E}[e^{-(Z_{\tau} - a)}| \tau \geq \Gamma](1-\mathrm{P}(\tau < \Gamma))
+ \mathrm{E}[e^{-(Z_{\tau} - a)}|  \tau < \Gamma] \mathrm{P}(\tau < \Gamma).\]
Since, $\mathrm{P}(\tau < \Gamma) = \mathrm{E}[1-p_{\tau}] \leq 1-A \leq e^{-a}$, we can write,
\[\mathrm{PFA} = e^{-a} \mathrm{E}[e^{-(Z_{\tau} - a)}|  \tau \geq \Gamma](1+o(1)) \ \ \ \ \ \  \mbox{ as } a \ \to \infty.\]
This proves the lemma.
\end{IEEEproof}

\section*{Appendix to Section \ref{sec:Delay}}

\begin{IEEEproof}[Proof of Theorem \ref{thm:ADDExact}]
Each time $Z_k$ crosses $b$ from below, is satisfies
\[b \ < \ \ Z_k \ \ \leq \ \ b + \log\frac{1}{1-\rho}  + \log(1+e^{-b}\rho).\]
Define, $b_1 \stackrel{\triangle}{=} b + \log\frac{1}{1-\rho}  + \log(1+e^{-b}\rho)$. Then $b_1 \to b \mbox{ as }  \rho \to 0$.
Also, each time $Z_k$ crosses $b$ from below, the average time for $Z_k$ to reach $a$ can be decreased by setting $Z_k=b_1$ and increased by setting $Z_k=b$.
Let, $N$ ($N_1$) be one plus the number of times $Z_k$ goes below $b$ before it
crosses $a$, when it is reset to $b$ ($b_1$), each time it crosses $b$ from below.

Now recall the three disjoints events:
\begin{eqnarray*}
\mathcal{A} &=& \{Z_\Gamma<b\},\\
\mathcal{B} &=& \{Z_\Gamma\geq b; Z_k \nearrow b\},\\
\mathcal{C} &=& \{Z_\Gamma\geq b; Z_k \nearrow a\}.
\end{eqnarray*}
We can write,
\begin{eqnarray}
\label{eq:contionalDelayABC}
\mathrm{E}[\tau-\Gamma| \tau \geq \Gamma] &=& \mathrm{E}[\tau-\Gamma; \mathcal{A} | \tau \geq \Gamma]
+ \mathrm{E}[\tau-\Gamma;  \mathcal{B}| \tau \geq \Gamma]
+ \mathrm{E}[\tau-\Gamma; \mathcal{C}| \tau \geq \Gamma].
\end{eqnarray}
Now consider each of the three terms on the right hand side of the above equation.

Under the event $\mathcal{A}$, the process $Z_k$
starts below $b$ and reaches $a$ after multiple up-crossings of the threshold $b$.
Then,
\begin{equation}
\label{eq:CondDelayUB}
\mathrm{E}[\tau-\Gamma; \mathcal{A}| \tau \geq \Gamma] \leq \mathrm{E}[t(Z_\Gamma,b)| \mathcal{A}, \tau \geq \Gamma] \ \ \mathrm{P}(\mathcal{A}| \tau \geq \Gamma)
 + \mathrm{E}_1\left[ \sum_{k=1}^{N} \Lambda_k(b)\right]
\mathrm{P}(\mathcal{A}| \tau \geq \Gamma).
\end{equation}
This upper bound was obtained by resetting $Z_k$ to $b$ each time it crosses $b$ from below. Similarly, we can get a lower bound by setting $Z_k=b_1$ each time $Z_k$ crosses $b$ from below. Thus,
\begin{equation*}
\mathrm{E}[\tau-\Gamma; \mathcal{A}| \tau \geq \Gamma] \geq \mathrm{E}[t(Z_\Gamma,b)| \mathcal{A}, \tau \geq \Gamma] \ \ \mathrm{P}(\mathcal{A}| \tau \geq \Gamma)
 + \mathrm{E}_1\left[ \sum_{k=1}^{N_1} \Lambda_k(b_1)\right]
\mathrm{P}(\mathcal{A}| \tau \geq \Gamma).
\end{equation*}

Now by Wald's lemma \cite{Siegmund},
\begin{eqnarray*}
\mathrm{E}_1\left[ \sum_{k=1}^{N_1} \Lambda_k(b_1)\right] &=& \mathrm{E}_1[N_1] \mathrm{E}_1[\Lambda(b_1)]\\
                                                          & \xrightarrow[\rho\to 0]{}& \mathrm{E}_1[N] \mathrm{E}_1[\Lambda(b)]
                                                           \hspace{0.2cm}= \hspace{0.2cm}\mathrm{E}_1\left[ \sum_{k=1}^{N} \Lambda_k(b)\right] = \mathrm{ADD}^s.
                                                          \end{eqnarray*}
Thus,
\begin{equation*}
\begin{split}\mathrm{E}[\tau-\Gamma; \mathcal{A}| \tau \geq \Gamma]  =& \bigg[ \mathrm{E}[t(Z_\Gamma,b)| \mathcal{A}, \tau \geq \Gamma]
\ \ \mathrm{P}(\mathcal{A}| \tau \geq \Gamma) \bigg. \\
 &+ \bigg. \mathrm{ADD}^s  \ \ \mathrm{P}(\mathcal{A}| \tau \geq \Gamma) \bigg] \big(1+o(1)\big) \mbox{ as } \rho\to 0.
 \end{split}
 \end{equation*}
Under the event $\mathcal{B}$, the process $Z_k$ starts above $b$ and crosses $b$ before $a$. It then has multiple
up-crossings of $b$, similar to the case of event $\mathcal{A}$. Arguing in a similar manner, we get
\begin{equation*}
\begin{split}\mathrm{E}[\tau-\Gamma; \mathcal{B} | \tau \geq \Gamma]   =&
\bigg[ \mathrm{E}[\Lambda(Z_\Gamma) | \mathcal{B}, \tau \geq \Gamma]  \ \ \mathrm{P}(\mathcal{B}| \tau \geq \Gamma) \bigg. \\
 &+ \bigg. \mathrm{ADD}^s  \ \ \mathrm{P}(\mathcal{B}| \tau \geq \Gamma) \bigg] \big(1+o(1)\big) \mbox{ as } \rho\to 0.
 \end{split}
 \end{equation*}
Similarly, considering the event $\mathcal{C}$, we get
\begin{equation*}
\mathrm{E}[\tau-\Gamma; \mathcal{C} | \tau \geq \Gamma]
= \bigg[ \mathrm{E}[\Lambda(Z_\Gamma) | \mathcal{C}, \tau \geq \Gamma] \ \ \mathrm{P}(\mathcal{C}| \tau \geq \Gamma) \bigg] \big(1+o(1)\big) \mbox{ as } \rho\to 0.
 \end{equation*}

Substituting in \eqref{eq:contionalDelayABC} we get the desired result \eqref{eq:CondDelayA}.
\end{IEEEproof}

\begin{IEEEproof}[Proof of Lemma \ref{eq:ADDsENuEqui}]
Based on $\Psi$, we define two new recursions, one in which the evolution of $Z_k$ is truncated at $b$,
\begin{eqnarray*}
\label{eq_appen:NewRecur1}
\tilde{\Psi}(Z_k) = \left\{ \begin{array}{ll}
               \Psi(Z_k) \ \ \ \  \mbox{ if } \Psi(Z_k) \geq b\\
               b         \ \ \ \ \ \ \ \ \ \ \ \mbox{ if } \Psi(Z_k) < b,
               \end{array}       \right.
\end{eqnarray*}
and, another in which the overshoot is ignored each time the Shiryaev recursion crosses $b$ from below,
\begin{eqnarray*}
\label{eq_appen:NewRecur2}
\hat{\Psi}(Z_k) = \left\{ \begin{array}{ll}
               b \ \ \ \ \ \ \ \ \ \ \ \ \  \mbox{ if } Z_k < b \mbox{ and } \Psi(Z_k) \geq b\\
               \Psi(Z_k)         \ \ \ \ \ \ \mbox{ otherwise }.
               \end{array}       \right.
\end{eqnarray*}
Based on these two recursions we define two new stopping times:
\begin{eqnarray*}
\tilde{\nu}_b &=& \inf\{k\geq 1:  \tilde{\Psi}(Z_{k-1}) > a, Z_0 = b\},\\
\hat{\nu}_b &=& \inf\{k\geq 1:  \hat{\Psi}(Z_{k-1}) > a, Z_0 = b\}.
\end{eqnarray*}
These two stopping times stochastically upper and lower bound the Shiryaev stopping time $\nu_b$ defined in \eqref{eq:ShiryaevNuBtoC}, i.e.,
\begin{eqnarray}
\label{eq:ENUInequalities}
\mathrm{E}_1[\tilde{\nu}_b] \hspace{-0.2cm}&\leq&  \hspace{-0.2cm}\mathrm{E}_1[\nu_b] \leq \mathrm{E}_1[\hat{\nu}_b].
\end{eqnarray}
Recall from \eqref{eq:nuxy} that
\[\nu(x,y) = \inf\{k \geq 1: \Psi(Z_{k-1}) > y, Z_0 = x\}.\]
Using Wald's lemma \cite{Siegmund}, we can get the following expressions:
\begin{equation}
\label{eq:NutildeNuCap}
\mathrm{E}_1[\tilde{\nu}_b] = \frac{\mathrm{E}_1[\lambda]}{\mathrm{P}_1(Z_\lambda>a)}, \hspace{2cm} \mathrm{E}_1[\hat{\nu}_b] = \frac{\mathrm{E}_1[\lambda] + \mathrm{E}_1[\nu(Z_\lambda, b) ; \{Z_\lambda < b\}] }{\mathrm{P}_1(Z_\lambda > a)}.
\end{equation}
Multiplying and dividing $\mathrm{ADD}^s$ by $\mathrm{E}_1[\lambda]$ we get
\begin{eqnarray*}
\mathrm{ADD}^s 
&=& \frac{\mathrm{E}_1[\lambda] + \mathrm{E}_1[t(Z_\lambda, b) ; \{Z_\lambda < b\}]}{\mathrm{E}_1[\lambda]} \frac{\mathrm{E}_1[\lambda]}{\mathrm{P}_1(Z_\lambda > a)} \\
&=& \mathrm{E}_1[\tilde{\nu}_b] \frac{\mathrm{E}_1[\lambda] + \mathrm{E}_1[t(Z_\lambda, b) ; \{Z_\lambda < b\}]}{\mathrm{E}_1[\lambda]}\\
&=& \mathrm{E}_1[\tilde{\nu}_b] (1 + o(1))\ \ \ \ \mbox{ as } a\to \infty.
\end{eqnarray*}
The last equality follows because $\Expect_1[\lambda]\to \infty$ as $a\to \infty$, while $\Expect_1[t(Z_\lambda, b) ; \{Z_\lambda < b\}]$ is not a function of $a$.
Similarly, multiplying and dividing $\mathrm{ADD}^s$ by $\mathrm{E}_1[\lambda] + \mathrm{E}_1[\nu(Z_\lambda, b) ; \{Z_\lambda < b\}]$ we get
\begin{eqnarray*}
\mathrm{ADD}^s = \mathrm{E}_1[\hat{\nu}_b] \left(1 + o(1) \right) \ \ \ \ \mbox{ as } \ \ a\to \infty.
\end{eqnarray*}
Using these two expressions for $\mathrm{ADD}^s$ and the relationship that
$\mathrm{E}_1[\tilde{\nu}_b] \leq \mathrm{E}_1[\nu_b] \leq \mathrm{E}_1[\hat{\nu}_b]$, we have,
\[\mathrm{ADD}^s = \mathrm{E}_1[\nu_b](1 + o(1)) \mbox{ as } a \to \infty.\]
\end{IEEEproof}

\begin{IEEEproof}[Proof of Theorem \ref{lem:ADD_1}]
Consider the upper bound \eqref{eq:CondDelayUB}:
\begin{equation*}
\mathrm{E}[\tau-\Gamma; \mathcal{A}| \tau \geq \Gamma] \leq \mathrm{E}[t(Z_\Gamma,b)| \mathcal{A}, \tau \geq \Gamma] \ \ \mathrm{P}(\mathcal{A}| \tau \geq \Gamma)
 + \mathrm{ADD}^s \ \ \mathrm{P}(\mathcal{A}| \tau \geq \Gamma).
\end{equation*}
Similarly, the upper bounds corresponding to the other two events $\mathcal{B}$ and $\mathcal{C}$ are:
\begin{eqnarray*}
\mathrm{E}[\tau-\Gamma; \mathcal{B} | \tau \geq \Gamma] \leq \mathrm{E}[\Lambda(Z_\Gamma) | \mathcal{B}, \tau \geq \Gamma] \ \mathrm{P}(\mathcal{B}| \tau \geq \Gamma) +
\mathrm{ADD}^s \ \ \mathrm{P}(\mathcal{B}| \tau \geq \Gamma)
\end{eqnarray*}
and,
\begin{eqnarray*}
\mathrm{E}[\tau-\Gamma; \mathcal{C}| \tau \geq \Gamma] &=& \mathrm{E}[\Lambda(Z_\Gamma) | \mathcal{C}, \tau \geq \Gamma] \ \ \mathrm{P}(\mathcal{C}| \tau \geq \Gamma) \\
&\leq& \mathrm{E}_1[\Lambda(b) | Z_{\Lambda(b)} >a] \ \ \mathrm{P}(\mathcal{C}| \tau \geq \Gamma) \\
&\leq& \mathrm{ADD}^s \ \ \mathrm{P}(\mathcal{C}| \tau \geq \Gamma).
\end{eqnarray*}
Substituting in \eqref{eq:contionalDelayABC} we get,
\begin{eqnarray}
\label{eq:tauCminusGammma}
\mathrm{E}[\tau-\Gamma| \tau \geq \Gamma] &=& \mathrm{E}[\tau-\Gamma; \mathcal{A} | \tau \geq \Gamma] + \mathrm{E}[\tau-\Gamma;  \mathcal{B}| \tau \geq \Gamma]
+ \mathrm{E}[\tau-\Gamma; \mathcal{C}| \tau \geq \Gamma]. \nonumber \\
&\leq&  \mathrm{ADD}^s  + \mathrm{E}[t(Z_\Gamma,b)| \mathcal{A}, \tau \geq \Gamma] + \mathrm{E}[\Lambda(Z_\Gamma) | \mathcal{B}, \tau \geq \Gamma].
\end{eqnarray}

In equation (\ref{eq:tauCminusGammma}), we observe that except for $\mathrm{ADD}^s$, other terms are not a
function of threshold $a$. Thus we have
\[\mathrm{E}[\tau-\Gamma| \tau \geq \Gamma] \leq \mathrm{ADD}^s\left(1 + o(1) \right) \ \  \mbox{ as } a \ \to \infty.\]
\end{IEEEproof}

\begin{IEEEproof}[Proof of Lemma \ref{lem:txyExprn_ADD}]
First note that by definition \eqref{eq:tXY}, $Z_{t(x,y)} > y \geq Z_{t(x,y)-1}$. Also, from (\ref{eq:ZkrecursionSkip})
\begin{eqnarray*}
Z_{t(x,y)} &=& Z_{t(x,y)-1}  + \log\frac{1}{1-\rho} + \log(1+e^{-Z_{t(x,y)-1}}\rho) \\
&\leq& y + \log\frac{1}{1-\rho}  + \log(1+e^{-y}\rho).
\end{eqnarray*}
Thus
\[y<Z_{t(x,y)} \leq y + \log\frac{1}{1-\rho}  + \log(1+e^{-y}\rho), \]
equivalently
\[e^y<e^{Z_{t(x,y)}} \leq e^y \frac{1}{1-\rho}  (1+e^{-y}\rho).\]
Further, the recursion (\ref{eq:ZkrecursionSkip}) can be written in terms of $e^{Z_k}$ for $k\geq0$:
\[e^{Z_{k+1}} = \frac{\rho + e^{Z_k}}{1-\rho}.\]
Using this we can write an expression for $e^{Z_{t(x,y)}}$:
\[e^{Z_{t(x,y)}} = \frac{e^x}{(1-\rho)^t} + \sum_{k=1}^{t(x,y)} \frac{\rho}{(1-\rho)^k} = \frac{e^x+1}{(1-\rho)^{t(x,y)}} - (1-\rho).\]
Using the bounds for $Z_{t(x,y)}$ obtained above, we get
\[ e^y < \frac{e^x+1}{(1-\rho)^{t(x,y)}} - (1-\rho)\leq e^y \frac{1}{1-\rho}  (1+e^{-y}\rho).\]
This gives us bounds for $t(x,y)$:
\begin{equation}
\label{eq:tXYbounds}
\frac{\log(1+e^y-\rho) - \log(1+e^x)}{|\log(1-\rho)|} \leq t(x,y) \leq \frac{\log\left(1+e^y\frac{(1+e^{-y}\rho)}{(1-\rho)}-\rho\right) - \log(1+e^x)}{|\log(1-\rho)|}.
\end{equation}
By keeping $x,y$ fixed and taking $\rho\to0$ we get \eqref{eq:tXYExprn_ANO}.
\end{IEEEproof}

\vspace{0.5cm}


\section*{Appendix to Section \ref{sec:Energy}}

\begin{IEEEproof}[Proof of Lemma \ref{lem:ANOBeforeGamma}]
Each time $Z_k$ crosses $b$ from below, is satisfies:
\[b \ < \ \ Z_k \ \ \leq \ \ b + \log\frac{1}{1-\rho}  + \log(1+e^{-b}\rho).\]
Define, $b_1 \stackrel{\triangle}{=} b + \log\frac{1}{1-\rho}  + \log(1+e^{-b}\rho)$. Then $b_1 \to b \mbox{ as }  \rho \to 0$.
Also, each time $Z_k$ crosses $b$ from below, the average number of observations used before $\Gamma$ can be increased by setting $Z_k=b_1$ and decreased by setting $Z_k=b$.
This is because of the geometric nature of change. Let $Z_k=x$ when it crosses $b$ from below, and suppose we reset $Z_k$ to $b_1$.
Then, the number of observations used before change, on an average, would be the number of observations used before $Z_k$ reaches $x$ from $b_1$,
plus the number of observations used there onwards as if the process started at $x$. Similar reasoning can be given to explain why the average
number of observations used decreases, if we reset $Z_k$ to $b$, each time it crosses $b$ from below.

Define the following stopping time:
\[\hLam^x = \inf\{k \geq 1: Z_{k-1}<b \mbox{ and } Z_k\geq b \mbox{ or } k\geq \Gamma, Z_0=x \geq b, a=\infty \}.\]
Thus, $\hLam^x$ is the time for $Z_k$, to start at $Z_0=x$ with $a=\infty$, and stop the first time, either $Z_k$ approaches $b$ from below, or when change happens.
Also, let $\delta^x \in (0,1)$ be such that $\hLam^x \delta^x$ is the number of observations used before $Z_k$ was stopped by $\hLam^x$, i.e., fraction of $\hLam^x$ when $Z_k\geq b$.
If $\{\hLam_k^b\}$ and $\{\hLam_k^{b_1}\}$ be sequences with distribution of $\hLam^b$ and $\hLam^{b_1}$ respectively and if $L^x$ is the number of times $Z_k$
crosses $b$ from below and is set to $x$ at each such instant, then,
\begin{eqnarray*}
\mathrm{E}_\infty[L^b] \ \ \mathrm{E}_\infty[\hLam^b \delta^b] = \mathrm{E}_\infty\left[ \sum_{k=1}^{L^{b}} \hLam_k^{b} \delta_k^{b} \right]
&\leq& \mathrm{E}\left[\sum_{k=t(b)}^{\Gamma-1} S_k \bigg| \Gamma > t(b), a=\infty \right] \\
&\leq&\mathrm{E}_\infty\left[ \sum_{k=1}^{L^{b_1}} \hLam_k^{b_1} \delta_k^{b_1} \right] = \mathrm{E}_\infty[L^{b_1}] \ \ \mathrm{E}_\infty[\hLam^{b_1} \delta^{b_1}].
\end{eqnarray*}
Here the equalities follows from Wald's lemma \cite{Siegmund}.

In the above, $L^{x}$ is $\mathrm{Geom}(\mathrm{P}_\infty[\Gamma \leq \hLam^x])$, and hence $\mathrm{E}_\infty[L^{b_1}] = \frac{1}{\mathrm{P}_\infty[\Gamma \leq \hLam^{b_1}]}$.
Also note that
\[\frac{\mathrm{P}_\infty[\Gamma \leq \hLam^{b_1}]}{\mathrm{P}_\infty[\Gamma \leq \hLam^b]} \to 1 \mbox{ as }  \rho \to 0.\]
Further, for $x=b_1$ or $x=b$, define $\hat{\lambda}(x)$ based on \eqref{eq:LambdaInfty} as
\[\hat{\lambda}(x) = \inf\{k \geq 1: Z_k < b, Z_0=x\geq b, a=\infty\}.\]
It is clear that $\hat{\lambda}(b)=\hat{\lambda}$. Thus we have, for both $x=b_1$ and $x=b$,
\begin{eqnarray*}
\mathrm{E}_\infty[\hLam^x \delta^x] &=& \mathrm{E}_\infty[\hLam^x \delta^x|\Gamma \leq \hLam^x \delta^x]\mathrm{P}_\infty[\Gamma \leq \hLam^x\delta^x] + \mathrm{E}_\infty[\hLam^x \delta^x| \Gamma > \hLam^x \delta^x]\mathrm{P}_\infty[\Gamma > \hLam^x\delta^x]\\
                      &\to& \mathrm{E}_\infty[\hat{\lambda}(x)] \mbox{ as }  \rho \to 0.
\end{eqnarray*}
Here, the result follows because as $\rho \to 0$, $\hLam^x \delta^x$ converges a.s. to a finite limit and $\mathrm{P}_\infty[\Gamma \leq \hLam^x\delta^x] \to 0$. Also for the same reason,
$\mathrm{P}_\infty[\Gamma > \hLam^x\delta^x] \to 1$ as $\rho\to 0$. Moreover, since $b_1 \to b$ as $\rho\to 0$, we have as $\rho\to 0$
\[\mathrm{E}_\infty[\hat{\lambda}(b_1)] \to \mathrm{E}_\infty[\hat{\lambda}(b)] = \mathrm{E}_\infty[\hat{\lambda}].\]
Thus,
\[\mathrm{E}\left[\sum_{k=t(b)}^{\Gamma-1} S_k\Bigg| \Gamma > t(b), a=\infty \right]=\frac{\mathrm{E}_\infty[\hat{\lambda}]}{\mathrm{P}_\infty[\Gamma \leq \hLam^b]}(1+o(1)) \ \ \ \mbox{ as }  \rho \to 0. \]
\end{IEEEproof}

\begin{IEEEproof}[Proof of Lemma \ref{lem:ANOProb}]
Since $\Prob\{ \tau \geq \Gamma\} \to 1$ as $a \to \infty$,
\begin{eqnarray*}
\mathrm{P}(\Gamma > t(b)| \tau \geq \Gamma ) &=&  \mathrm{P}(\Gamma > t(b))  + o(1) \ \ \ \mbox{ as } \ \ a \to \infty \\
&=& \frac{1}{1+z_0}(1-\rho)^{t(b)} + o(1) \ \ \ \mbox{ as } \ \ a \to \infty.
\end{eqnarray*}
From \eqref{eq:tXYExprn_ANO} in Lemma \ref{lem:txyExprn_ADD}, with $y=b$ and $x=z_0$, we have
\begin{equation*}
t(z_0,b) = \left(\frac{\log(1+e^b)-\log(1+e^{z_0})}{|\log(1-\rho)|}\right)(1 + o(1)) \mbox{ as } \rho \to 0.
\end{equation*}
From this, it is easy to show that
\begin{eqnarray*}
(1-\rho)^{t(b)} \to \left(\frac{1+e^{z_0}}{1+e^b}\right) \ \ \ \mbox{ as } \rho \to 0.
\end{eqnarray*}
By substituting this in the expression for $\mathrm{P}(\Gamma > t(b)| \tau \geq \Gamma )$ we get the desired result.
\end{IEEEproof}

\begin{IEEEproof}[Proof of Theorem \ref{lem:ANO1}]
Using Theorem \ref{lem:ADD_1} we write $\mathrm{ANO}_1$ as
\begin{eqnarray*}
\mathrm{ANO}_1 &=& \Expect\left[\tau-\Gamma| \tau \geq \Gamma\right]\left(1-\frac{T_b-1}{\mathrm{E}\left[\tau-\Gamma| \tau \geq \Gamma\right]}\right)\\
               &=& \Expect_1[\nu_b]\left(1-\frac{T_b-1}{\mathrm{E}\left[\tau-\Gamma| \tau \geq \Gamma\right]}\right) (1 + o(1)) \ \ \ \mbox{ as } \ \ a \to \infty.
\end{eqnarray*}
We now obtain an upper bound on $\frac{T_b-1}{\mathrm{E}\left[\tau-\Gamma| \tau \geq \Gamma\right]}$ which goes to zero as $a \to \infty$.

Recall that $\mathcal{A}$ and $\mathcal{B}$ are the events under which excursions below $b$ are possible. The passage to $a$
is through multiple cycles below $b$, and the time spend below $b$ in each cycle can be bounded by $t(-\infty, b)$. Define $N_{\mathcal{A}}$ and $N_{\mathcal{B}}$ as
one plus the number of cycles below $b$, under events $\mathcal{A}$ and $\mathcal{B}$ respectively. Then,
\[T_b-1 \leq T_b \leq  \mathrm{P}_1(\mathcal{A}) t(-\infty,b)\mathrm{E}[N_{\mathcal{A}}] +   \mathrm{P}_1(\mathcal{B}) t(-\infty,b)\mathrm{E}[N_{\mathcal{B}}].\]
The averages $\mathrm{E}[N_{\mathcal{A}}]$ and $\mathrm{E}[N_{\mathcal{B}}]$
can be written as a series of probabilities, where each term correspond to the event that $Z_k$ goes below $b$, and not above $a$, each time it crosses $b$ from below.
Each of these probabilities can be maximized by setting $Z_k$ to $b$, each time it crosses $b$ from below. Hence, $\mathrm{E}[N_{\mathcal{A}}]\leq \mathrm{E}[N]$ and
$\mathrm{E}[N_{\mathcal{B}}]\leq \mathrm{E}[N]$. This gives a bound on $T_b-1$.
\[T_b-1 \leq t(-\infty,b)\mathrm{E}[N].\]
By using \eqref{eq:ENUInequalities} we get as $a \to \infty$,
\begin{eqnarray*}
\frac{T_b-1}{\mathrm{E}\left[\tau-\Gamma| \tau \geq \Gamma\right]} \ \leq \ \frac{t(-\infty,b)\mathrm{E}[N]}{\mathrm{E}_1[\nu_b]} (1 + o(1)) \ \leq \ \frac{t(-\infty,b)\mathrm{E}[N]}{\mathrm{E}_1[\tilde{\nu}_b]} (1 + o(1)).
\end{eqnarray*}
From \eqref{eq:NutildeNuCap} we know that $\mathrm{E}_1[\tilde{\nu}_b]=\mathrm{E}_1[\lambda]\mathrm{E}[N]$. Thus the upper bound on $\frac{T_b-1}{\mathrm{E}\left[\tau-\Gamma| \tau \geq \Gamma\right]}$ goes to 0 as $a \to \infty$. This proves the theorem.
\end{IEEEproof}

\end{document}